\documentclass[11pt]{amsart}

\usepackage[margin=1.05in]{geometry}
\usepackage{amsmath,amssymb,amsthm,mathtools,mathrsfs}
\usepackage{microtype}
\usepackage{enumitem}
\usepackage{booktabs}
\usepackage{xcolor}
\usepackage[colorlinks=true,
  linkcolor=red,
  citecolor=blue!55!black,
  urlcolor=blue!55!black]{hyperref}
\usepackage{aliascnt}

\numberwithin{equation}{section}

\newtheorem{theorem}{Theorem}[section]
\newaliascnt{proposition}{theorem}
\newtheorem{proposition}[proposition]{Proposition}
\aliascntresetthe{proposition}
\newaliascnt{lemma}{theorem}
\newtheorem{lemma}[lemma]{Lemma}
\aliascntresetthe{lemma}
\newaliascnt{corollary}{theorem}

\aliascntresetthe{corollary}
\newaliascnt{remark}{theorem}
\newtheorem{remark}[remark]{Remark}
\aliascntresetthe{remark}
\theoremstyle{definition}
\newaliascnt{definition}{theorem}

\aliascntresetthe{definition}

\newcommand{\B}{\mathbb B}
\newcommand{\PP}{\mathbb P}
\newcommand{\CC}{\mathbb C}
\newcommand{\OO}{\mathcal O}
\newcommand{\II}{\mathcal I}
\newcommand{\AG}{\mathcal A_G}
\newcommand{\JJ}{\mathcal J}
\newcommand{\TT}{\mathcal T}
\newcommand{\TTzero}{\mathcal T_0}
\newcommand{\QQ}{\mathcal Q}
\newcommand{\EE}{\mathscr E}
\newcommand{\LL}{\mathscr L}
\newcommand{\tors}{\operatorname{tors}}
\newcommand{\ord}{\operatorname{ord}}
\newcommand{\Span}{\operatorname{span}}
\newcommand{\rank}{\operatorname{rank}}
\newcommand{\coker}{\operatorname{coker}}

\newcommand{\ch}{\operatorname{ch}}

\newcommand{\Sym}{\operatorname{Sym}}
\newcommand{\diag}{\operatorname{diag}}

\newcommand{\pr}{\operatorname{pr}}
\newcommand{\Ann}{\operatorname{Ann}}
\newcommand{\Wr}{\operatorname{Wr}}
\newcommand{\length}{\operatorname{length}}

\title[Sharp degree bound for rational proper maps from $\B^2$ to $\B^4$]
{Sharp degree bound for rational proper maps \\from $\B^2$ to $\B^4$}
\author[Tianzhi Hu]{Tianzhi Hu}
\address{ School of Mathematics and Statistics, Wuhan University, Luojiashan, Wuchang, Wuhan, Hubei, 430072, P.R. China}
\email{hutianzhi@whu.edu.cn}

\author[Mai Shi]{Mai Shi}
\address{ School of Mathematics and Statistics, Wuhan University, Luojiashan, Wuchang, Wuhan, Hubei, 430072, P.R. China}
\email{shimai@whu.edu.cn}

\author[Pingsan Yuan]{Pingsan Yuan}
\address{ School of Mathematics and Statistics, Wuhan University, Luojiashan, Wuchang, Wuhan, Hubei, 430072, P.R. China}
\email{pingsanyuan@whu.edu.cn}

\date{}

\begin{document}

\begin{abstract}
We prove that every rational proper holomorphic map from $\mathbb{B}^2$ to $\mathbb{B}^4$ has degree at most $5$, confirming D’Angelo's conjecture in this case. This sharp bound follows from Meylan’s degree estimates and an intersection-theoretic argument ruling out degree $6$.
\end{abstract}


\maketitle
\setcounter{tocdepth}{1}
\tableofcontents

\section{Introduction}
Let $\B^n=\{z\in\CC^n:\|z\|<1\}$ be the unit ball.
Proper holomorphic maps between balls and their rigidity properties
are important topics in several complex variables and CR geometry.
Two maps $f,g:\B^n\to\B^N$ are said to be equivalent
if $f=\tau\circ g\circ\sigma$ for some
$\sigma\in\operatorname{Aut}(\mathbb B^n)$ and
$\tau\in\operatorname{Aut}(\mathbb B^N)$.
In equal dimensions, Alexander \cite{Alexander1977} proved that every
proper holomorphic self-map of $\B^n$ with $n\ge2$ is an automorphism.

For maps between balls of different dimensions, Webster
\cite{Webster1979} proved that a proper holomorphic map
from $\B^n$ to $\B^{n+1}$, with $n\ge3$, extending
$C^3$-smoothly to the boundary is equivalent to the standard
linear embedding. Huang \cite{Huang1999} extended rigidity result to proper holomorphic maps from $\B^n$ to $\B^N$ with
$n<N\le2n-2$, assuming only $C^2$ boundary regularity.

Boundary regularity also plays an important role in the rationality of proper holomorphic maps.
Forstneri\v{c} \cite{Forstneric1989} proved that a proper
holomorphic map from $\B^n$ to $\B^N$ with $2\le n<N$, is rational if it extends $C^{N-n+1}$-smoothly to the boundary.
Huang \cite{Huang2001} showed that $C^2$ boundary regularity
suffices when $N\le2n-1$. These rationality results provide a natural framework for the study of rational proper
holomorphic maps between balls.

Classical classification results for rational proper holomorphic maps between balls include those of Faran \cite{Faran1982} for $\B^2\to\B^3$ and Huang--Ji \cite{HuangJi2001} for $\B^n\to\B^{2n-1}$ with $n\ge3$. Further classification results in higher codimensions
include \cite{Hamada2005,HJX2006}. More recent progress on normal forms of degree-three rational sphere maps in two variables can be found in \cite{DJGHLXZ2026}.
Another important direction is the \emph{gap phenomenon}, concerning ranges of target dimensions in which every rational proper holomorphic map from $\B^n$ to $\B^N$ is equivalent to a map of the form $(g,0)$, where $g$ takes values in a strictly lower-dimensional ball. Relevant results include \cite{HJX2006,HJY2014,GaoNg2024,YinYuan2026}.

A fundamental problem in the study of rational proper holomorphic maps between balls is to establish sharp bounds for their algebraic degrees. Let $f:\mathbb B^n\to\mathbb B^N$ with $N\ge n\ge2$, be a rational proper holomorphic map written in reduced form as 
\[
f=\frac{(P_1,\ldots,P_N)}{Q},
\qquad
\gcd(P_1,\ldots,P_N,Q)=1,
\]
where $P_1,\ldots,P_N,Q\in\mathbb C[z_1,\ldots,z_n]$. Its algebraic degree is defined by
\[
\deg f:=\max\{\deg P_1,\ldots,\deg P_N,\deg Q\}.
\]
This degree is invariant under equivalence.

The degree estimate problem has been studied extensively by D'Angelo \cite{DAngelo1988,DAngeloBook} and D'Angelo-Lebl
\cite{DAngeloLebl2009}. D'Angelo's degree conjecture asserts that every rational proper holomorphic map $f:\B^n\to\B^N$ satisfies
\begin{equation*}
\deg f\le
\begin{cases}
2N-3, & n=2,\\[2mm]
\dfrac{N-1}{n-1}, & n\ge 3.
\end{cases}
\end{equation*}

Several cases of the conjecture have been established through the study of geometric rank. Huang-Ji-Xu \cite{HJX2006}
proved the conjectured bound for maps of geometric rank one
when $n\ge3$. By analyzing maps of geometric rank
two, Ji-Yin \cite{JiYin2020} proved that every rational proper
holomorphic map from $\B^n$ to $\B^{4n-6}$ with $n\ge7$, has degree
at most three, confirming the conjecture in the third gap
interval and at its upper endpoint.

For $n=2$, D'Angelo-Kos-Riehl \cite{DKR2003} established the bound $2N-3$ for monomial proper maps and constructed examples attaining equality.
For general rational proper maps from $\B^2$ to $\B^3$, Faran's classification \cite{Faran1982} implies that those maps has degree at most three, and this bound is sharp. We consider the next target dimension, $N=4$, for which the conjectured bound is five.

Meylan \cite{Meylan2006} established the general estimate
\begin{equation*}
  \deg f\le\frac{N(N-1)}{2}
\end{equation*}
for rational proper holomorphic maps $f:\B^2\to\B^N$.
In particular, when $N=4$, this gives $\deg f\le6$.
Thus, to establish the conjectured bound for maps from $\B^2$ to $\B^4$, it remains to rule out the existence of maps of degree six. The main result of this paper is the following.

\begin{theorem}\label{thm:degree-six}
Let $f:\mathbb B^2\to\mathbb B^4$ be a rational proper
holomorphic map. Then
\[
\deg f\leq 5.
\]
Moreover, this bound is sharp.
\end{theorem}

\begin{remark}
Sharpness is already visible in the monomial category.  For example,
\begin{equation*}
(z,w)\longmapsto \bigl(z^5,\sqrt5\,z^3w,\sqrt5\,zw^2,w^5\bigr)
\end{equation*}
is proper from $\B^2$ to $\B^4$ and has degree five. This example belongs to the sharp family underlying the monomial theorem of \cite{DKR2003}.
\end{remark}

We sketch the proof of Theorem~\ref{thm:degree-six}.
In Section~\ref{sec:homogeneous}, we homogenize and complexify
the sphere identity of a hypothetical degree-six proper map,
obtaining the polarized incidence identity
\eqref{eq:polarized-identity}.
In Section~\ref{sec:projective}, we introduce the associated
relative third-jet map $\Phi_0:\AG\to\JJ$ in \eqref{eq:Phi}
and use intersection theory to compute a characteristic number
associated with its degeneracy; see \eqref{eq:global33}.
The global accounting formula \eqref{eq:account} expresses
this number in terms of base-point contributions and local defects.
Finally, in Section~\ref{sec:local-proof}, we establish
a local lower bound that is incompatible
with the global computation, yielding the desired contradiction.

\vspace{.5cm}

\noindent\textbf{AI Declaration:} The proofs of the main results in this paper arose from discussions and exchanges between the authors and AI. In addition, AI tools were also used for calculations and for polishing the writing. All
mathematical statements have been independently checked by the authors, who take full
responsibility for the content.

\section{Homogenization, complexification, and preliminary rigidity}\label{sec:homogeneous}

\subsection{Homogenization and complexification}\label{subsec:hom-complex}

We use the projective model of the ball.  Put
\[
J_{1,n}:=\diag(1,-1,\ldots,-1)
\]
and define the Hermitian form $\langle\ ,\ \rangle_{1,n}$ by
\[
\langle Z,U\rangle_{1,n}:=Z^tJ_{1,n}\overline U,
\qquad
\langle Z,Z\rangle_{1,n}=|Z_0|^2-|Z_1|^2-\cdots-|Z_n|^2.
\]
The ball $\B^n$ is identified with the positive projective lines in $\PP^n$, and its boundary with the projectivized null cone $\langle Z,Z\rangle_{1,n}=0$.

For independent complex variables $Z$ and $W$ we write
\begin{equation*}
(Z,W)_{1,n}:=Z^tJ_{1,n}W.
\end{equation*}
Thus $(Z,W)_{1,n}$ is the complex bilinear polarization of the Hermitian form $\langle\ ,\ \rangle_{1,n}$.  In particular,
\[
(Z,W)_{1,2}=Z_0W_0-Z_1W_1-Z_2W_2.
\]
This notation separates the Hermitian form $\langle\ ,\ \rangle_{1,n}$ from its algebraic complexification $(\ ,\ )_{1,n}$.

Assume, toward a contradiction, that
\[
f:\B^2\longrightarrow\B^4
\]
is rational proper of algebraic degree six.  Choose a reduced homogeneous lift
\begin{equation*}
F=(F_0,F_1,F_2,F_3,F_4):\CC^3\longrightarrow\CC^5,
\qquad \deg F_j=6,
\end{equation*}
whose entries have no common nonconstant factor.

For a polynomial $H$ in the $Z$-variables let
\[
H^*(W):=\overline{H(\overline W)},
\]
so that $H^*$ is obtained by conjugating the coefficients and replacing $Z$ by $W$.  We set
\begin{equation*}
G(W):=J_{1,4}F^*(W).
\end{equation*}
Then
\[
F(Z)^tG(W)=(F(Z),F^*(W))_{1,4}.
\]

A rational proper map between balls extends holomorphically across the source sphere; we use the boundary regularity theorem of Cima--Suffridge \cite{CimaSuffridge1990}.  Properness therefore gives
\[
\langle F(Z),F(Z)\rangle_{1,4}=0
\qquad\text{whenever}\qquad
\langle Z,Z\rangle_{1,2}=0.
\]
The same equality is automatic at a point where $F(Z)=0$.  Complexifying the boundary identity gives
\begin{equation*}
F(Z)^tG(W)=0
\qquad\text{on}\qquad
(Z,W)_{1,2}=0.
\end{equation*}
Equivalently, the complexified sphere identity preserves the corresponding Segre incidence relation.  Since the bilinear form $(Z,W)_{1,2}$ is irreducible, divisibility in the polynomial ring gives a bihomogeneous polynomial $R$ of bidegree $(5,5)$ such that
\begin{equation}\label{eq:polarized-identity}
F(Z)^tG(W)=(Z,W)_{1,2}R(Z,W).
\end{equation}

\subsection{Base-point freeness and linear fullness}\label{subsec:base-linear}

The following elementary consequence of boundary regularity is used later in the Hermitian-tangent case.

\begin{lemma}\label{lem:no-base}
The reduced homogeneous lift $F$ has no common zero over $\overline{\B^2}$ in the projective model.  The same holds for $G$ after coefficient conjugation.
\end{lemma}

\begin{proof}
By holomorphicity in the interior and the boundary regularity theorem of
Cima--Suffridge \cite{CimaSuffridge1990}, near every
$p\in\overline{\B^2}$ the projective map of $F$ admits a nonvanishing local
holomorphic lift $H=(H_0,\ldots,H_4)$.  After shrinking the neighborhood,
we may assume $H_0\neq0$ and write
\[
 F=aH
\]
for a holomorphic function $a$.  If $F(p)=0$, then $a(p)=0$.  Any local
irreducible hypersurface component of $(a=0)$ is contained in the common
zero set of the polynomials $F_0,\ldots,F_4$; its Zariski closure is a
codimension-one component of that common zero set.  Hence an irreducible
polynomial defining this component divides every $F_i$, contradicting the
reducedness of $F$.  Thus $F(p)\neq0$.  The assertion for $G$ follows by
coefficient conjugation.
\end{proof}

We next record a restriction lemma for plane linear systems.  

\begin{lemma}\label{lem:restriction}
Let $V\subset H^0(\PP^2,\OO(d))$ be a linear system.
\begin{enumerate}[label=\textup{(\roman*)}]
\item If $\dim V=5$ and the image of $V\to H^0(L,\OO_L(d))$ has dimension at most three for a general line $L$, then the members of $V$ have a common factor of degree at least $d-2$.
\item If $\dim V=4$ and the general restriction rank is at most two, or if $\dim V=3$ and the general restriction rank is at most one, then a system with no common factor cannot occur.
\item If $\dim V=2$, the general restriction cannot be zero.
\end{enumerate}
\end{lemma}

\begin{proof}
Remove the greatest common divisor $D$ of the system and put
$e:=d-\deg D$; we keep the notation $V$ for the residual system.  Its base
scheme is finite.  Choose a general line $L=(\ell=0)$ avoiding this base
scheme, put $U:=V|_L$, and define
\[
 K_j:=V\cap \ell^jH^0(\PP^2,\OO(e-j)).
\]
For a nearby line $L_t=(\ell+tm=0)$, the dimensions of the finitely many
spaces $K_j(L_t)$ are locally constant after choosing $L$ generally.
After shrinking the neighborhood of $L$ in the dual plane, these spaces
therefore form local vector subbundles.  If
$\ell^ja\in K_j\setminus K_{j+1}$, then, for any linear form $m$, we may
choose a local section
\[
s(t)=(\ell+tm)^ja_t\in V,
\qquad a_0=a.
\]
Since $s^{(j)}(0)\in V$, restriction to $L$ gives $s^{(j)}(0)|_L=j!\,m^j a|_L\in U,$ because every other term contains a positive power of $\ell$ and hence
vanishes on $L$.  As the restrictions $m|_L$ span
$H^0(L,\OO_L(1))$, and the $j$-th powers of linear forms span
$H^0(L,\OO_L(j))$, we obtain
\begin{equation}\label{eq:restriction-symbol}
 H^0(L,\OO_L(j))\,a|_L\subset U.
\end{equation}

For (i), assume $\dim V=5$ and $\dim U\le3$.  Then $\dim K_1\ge2$.
If $K_2\neq0$, choose $r\ge2$ maximal with $K_r\neq0$ and
$\ell^ra\in K_r\setminus K_{r+1}$.  By \eqref{eq:restriction-symbol}, $U$ contains a subspace of
dimension $r+1$, so $r=2$ and $\dim U=3$.  Hence
$U=a|_L H^0(L,\OO_L(2))$.  If $e>2$, the positive-degree form $a|_L$
has a zero, giving a base point of $U$, contrary to the choice of $L$.
Thus $e\le2$.

Suppose now that $K_2=0$.  Choose independent elements
$\ell a,\ell b\in K_1$.  Their symbols $a|_L,b|_L$ are independent, and
\eqref{eq:restriction-symbol} gives
\[
 H^0(L,\OO_L(1))\operatorname{span}\{a|_L,b|_L\}\subset U.
\]
The left-hand side has dimension at least three, hence equals $U$ and has
dimension three.  The corresponding multiplication map therefore has a
one-dimensional kernel, giving a nontrivial linear syzygy
$\ell_1a+\ell_2b=0$ on $L$.  The linear forms $\ell_1,\ell_2$ are
independent; by the UFD property of the binary polynomial ring,
$U=hH^0(L,\OO_L(2))$ for a form $h$ of degree $e-2$.  Again $e>2$ would
give a base point of $U$.  Thus $e\le2$, and in either case
$\deg D=d-e\ge d-2$.

For (ii), if $K_2\neq0$, the maximal-$r$ argument above produces a
subspace of $U$ of dimension at least three.  If $K_2=0$, two independent
first symbols do the same.  This contradicts $\dim U\le2$ when
$\dim V=4$, and a fortiori $\dim U\le1$ when $\dim V=3$.  Finally, in
(iii), restriction rank zero on a general line forces every member of
$V$ to vanish identically.  Hence $V=0$.
\end{proof}

\begin{lemma}\label{lem:linear-full}
The entries $F_0,\ldots,F_4$ are linearly independent.  Equivalently, the projective image is linearly full in $\PP^4$.
\end{lemma}

\begin{proof}
Suppose that the image is contained in a proper target projective
subspace, and let $\PP(M)$ be the smallest such subspace.  Since it
contains the image of an interior point of $\B^2$, the target Hermitian
form has a positive vector in $M$; its restriction to $M$ is therefore
nondegenerate of signature $(1,r)$, where $r=\dim\PP(M)\le3$.  Properness excludes $r=0$.  After a
target Hermitian isometry, $f$ is a linearly full proper map
\[
 \B^2\longrightarrow\B^r
\]
of the same reduced degree six.  If $r=3$, Meylan's estimate gives
$\deg f\le3$ \cite{Meylan2006}; if $r=2$, Alexander's theorem gives an
automorphism, hence degree one \cite{Alexander1977}.  The case $r=1$ is
impossible, since a proper map from the Stein domain $\B^2$ has compact
analytic fibers, whereas the dimension theorem gives a positive-dimensional
generic fiber.  Each case contradicts $\deg f=6$.
\end{proof}

\section{The projective-algebraic framework}\label{sec:projective}

The main proof is organized on the incidence variety of points and lines in the source projective plane.  This section introduces the global jet map, computes its total defect, and reduces the theorem to a local estimate along the determinant divisor.

\subsection{The incidence threefold and the sheaves $\mathcal A_G$ and $\mathcal J$}\label{subsec:AGJ}

Let
\begin{equation*}
\II=\bigl\{(Z,W)\in\PP^2_Z\times\PP^2_W:(Z,W)_{1,2}=0\bigr\},
\end{equation*}
be the incidence variety and 
\[
\pr_1:\II\to\PP^2_Z,
\qquad
\pr_2:\II\to\PP^2_W
\]
be the two projections. For \(p,q\in\mathbb Z\), set $\OO_\II(p,q)
:=
\pr_1^*\OO_{\PP^2_Z}(p)
\otimes
\pr_2^*\OO_{\PP^2_W}(q).$ We write
\[
h:=c_1\!\left(\OO_\II(1,0)\right),
\qquad
k:=c_1\!\left(\OO_\II(0,1)\right).
\]
The $\pr_2$-fiber over $W$ is the source line
\[
L_W=\{Z\in\PP^2:(Z,W)_{1,2}=0\}.
\]

Restricting \eqref{eq:polarized-identity} to $\II$ gives
\begin{equation}\label{eq:FG-zero-on-I}
F(Z)^tG(W)=0.
\end{equation}
The entries of $G$ have no common divisor, so their common zero set in $\PP^2_W$ has codimension at least two.  Define
\begin{equation}\label{eq:AG}
\AG:=\coker\!\left(
\OO_\II(0,-6)\xrightarrow{\ G\ }\OO_\II^{\oplus5}
\right).
\end{equation}
This is a rank-four torsion-free sheaf. 

The second sheaf records the Taylor data of degree-six sections along
the $\pr_2$-fibers.  For $W\in\PP^2_W$, let
$L_W=\pr_2^{-1}(W)\simeq\PP^1$.  At $(Z,W)\in\II$, the fiber of the
relative third jet bundle consists of local sections of
$\OO_\II(6,0)|_{L_W}$ modulo those vanishing to order at least four at
$Z$.  In a local coordinate $u$ on $L_W$, it is represented by
\[
\left(
s(Z),\,
\partial_us(Z),\,
\partial_u^2s(Z),\,
\partial_u^3s(Z)
\right).
\]
Thus these spaces form a rank-four vector bundle
\begin{equation*}
\JJ:=J_{\pr_2}^3\OO_\II(6,0).
\end{equation*}

\subsection{The relative third jet map and the global weight $33$}\label{subsec:global33}

There is a canonical relative jet-evaluation morphism
\[
\widetilde\Phi_0:\OO_\II^{\oplus5}\longrightarrow\JJ,
\qquad
e_i\longmapsto j_{\pr_2}^3(F_i),
\]
where $e_i$ is the $i$-th standard basis element of $\OO_\II^{\oplus5}$.  Because $G_i(W)$ is constant along a $\pr_2$-fiber and \eqref{eq:FG-zero-on-I} holds on $\II$,
\[
\sum_{i=0}^4G_i(W)j_{\pr_2}^3(F_i)=\sum_{i=0}^4j_{\pr_2}^3(G_i(W)F_i)=0.
\]
Hence $\widetilde\Phi_0$ descends to
\begin{equation}\label{eq:Phi}
\Phi_0:\AG\longrightarrow\JJ.
\end{equation}
On a line $L_W$, the transpose of \eqref{eq:Phi} is represented by the value and the first three directional derivatives of $F|_{L_W}$.

\begin{lemma}
 The determinant morphism $\det\Phi_0\in H^0(\II,\OO_\II(12,0))$ is a non-zero section.
\end{lemma}

\begin{proof}
Consider the relative fiber product
\[
  Y:=\II\times_{\PP^2_W} \II
\]
with projections $\pi_1,\pi_2\colon Y\longrightarrow \II$ and relative diagonal $\Delta\colon \II\hookrightarrow Y$. Let $\mathscr I_\Delta\subset\OO_Y$ be its ideal sheaf.  By definition,
\[
  J_{\pr_2}^3\mathcal{O}_\II(6,0)
  =
  \pi_{1*}\left(
    \pi_2^*\mathcal{O}_\II(6,0)\otimes
    \frac{\OO_Y}{\mathscr I_\Delta^4}
  \right).
\]

For $0\leq j\leq 4$, define
\[
  \mathcal{F}^j
  :=
  \pi_{1*}\left(
    \pi_2^*\mathcal{O}_\II(6,0)\otimes
    \frac{\mathscr I_\Delta^j}{\mathscr I_\Delta^4}
  \right),
\]
and the inclusions $\mathscr I_\Delta^{j+1}\subset
\mathscr I_\Delta^j$ give the filtration
\[
  J_{\pr_2}^3\mathcal{O}_\II(6,0)=\mathcal{F}^0\supset \mathcal{F}^1\supset \mathcal{F}^2
  \supset \mathcal{F}^3\supset \mathcal{F}^4=0.
\]
Thus we have
\[
  \frac{\mathcal{F}^j}{\mathcal{F}^{j+1}}
  \cong
  \pi_{1*}\left(
    \pi_2^*\mathcal{O}_\II(6,0)\otimes
    \frac{\mathscr I_\Delta^j}
         {\mathscr I_\Delta^{j+1}}
  \right).
\]
Therefore \[
  \frac{\mathcal{F}^j}{\mathcal{F}^{j+1}}
  \cong
  \mathcal{O}_\II(6,0)\otimes\operatorname{Sym}^j\Omega_{\II/\PP^2_W}
  =\mathcal{O}_\II(6,0)\otimes\Omega_{\II/\PP_W^2}^{\otimes j}
\]
by the canonical isomorphism $\frac{\mathscr I_\Delta}{\mathscr I_\Delta^2}
  \cong
  \Omega_{\II/\PP^2_W}$ and the relative adjunction:
\begin{equation*}
\Omega_{\II/\PP^2_W}\simeq\OO_\II(-2,1).
\end{equation*} 
Thus the grading pieces of $\mathcal F^\bullet$ are
\begin{equation}\label{eq:J-filtration}
\OO_\II(6,0),\qquad
\OO_\II(4,1),\qquad
\OO_\II(2,2),\qquad
\OO_\II(0,3).
\end{equation}

From \eqref{eq:AG} and \eqref{eq:J-filtration},
\[
\det\AG=\OO_\II(0,6),
\qquad
\det\JJ=\OO_\II(12,6).
\]
Therefore $\det\Phi_0\in H^0(\II,\OO_\II(12,0))
\simeq H^0(\PP^2_Z,\OO_{\PP^2_Z}(12))$, where the isomorphism is induced
by $\pr_1^*$.  We may write
\begin{equation*}
\det\Phi_0=\lambda(Z),
\qquad
\lambda\in H^0(\PP^2,\OO(12)).
\end{equation*}

The form $\lambda$ is nonzero.  Indeed, if the determinant vanished identically, then the Wronskian criterion would imply that the restriction to a general source line of the five-dimensional system generated by $F_0,\ldots,F_4$ has dimension at most three.  Lemma~\ref{lem:restriction}(i) would then give a common factor of degree at least four, contradicting the reducedness of $F$.  Thus $\Phi_0$ is generically an isomorphism. 
\end{proof}

Since $\AG$ is torsion-free, consider the exact sequence
\begin{equation*}
0\longrightarrow\AG\xrightarrow{\ \Phi_0\ }\JJ
\longrightarrow\TTzero\longrightarrow0.
\end{equation*}

\begin{lemma}
    We have the following identity
    \begin{equation}\label{eq:global33}
\int_\II h\,\ch_2(\TTzero)=33.
\end{equation}
\end{lemma}

\begin{proof}
By the projective bundle formula \cite[\S3.3]{Fulton},
its Chow ring satisfies
\begin{equation*}
h^2-hk+k^2=0,
\qquad
\int_\II h^2k=\int_\II hk^2=1.
\end{equation*}

In $K$-theory, from \eqref{eq:AG} and \eqref{eq:J-filtration},
\begin{align*}
[\AG]&=5[\OO_\II]-[\OO_\II(0,-6)],\\
[\JJ]&=[\OO_\II(6,0)]+[\OO_\II(4,1)]
+[\OO_\II(2,2)]+[\OO_\II(0,3)],
\end{align*}
where $[\cdot]$ denotes the class of a coherent sheaf in the Grothendieck
group $G_0(\II)$; see, for example, \cite[Chapter II]{Weibel2013}. 
Using $\ch(\OO_\II(a,b))=e^{ah+bk}$ 
(see, e.g., \cite[\S3.2]{Fulton}), we obtain
\[
\ch(\TTzero)
=
e^{6h}+e^{4h+k}+e^{2h+2k}+e^{3k}-5+e^{-6k}.
\]
Taking the degree-two component gives
\begin{align*}
\ch_2(\TTzero)=28h^2+8hk+25k^2=3h^2+33hk.
\end{align*}  It follows that
\[
\int_\II h\,\ch_2(\TTzero)=\int_\II 3h^3+33h^2k=33.  \qedhere
\]
\end{proof}

\subsection{The local transverse defect and reduction to the key estimate}\label{subsec:prop31}
Write the determinant divisor as
\begin{equation*}
(\lambda)=\sum_\alpha m_\alpha C_\alpha,
\qquad e_\alpha:=\deg C_\alpha.
\end{equation*}
On $\PP_W^2$ put
\[
E_G:=\operatorname{coker}\bigl(\OO(-6)\xrightarrow{G}\OO^{\oplus5}\bigr),
\qquad E:=E_G^{**}.
\]
Then $E$ is locally free.  Set
\[
\mathcal A:=\pr_2^*E,
\qquad
\QQ:=\mathcal A/\AG=\pr_2^*(E/E_G).
\]
Away from the codimension-two support of $\QQ$, the map \eqref{eq:Phi} extends from $\AG$ to $\mathcal A$; Hartogs extension gives a unique global morphism
\begin{equation*}
\Phi:\mathcal A\to\JJ,
\qquad
\TT:=\operatorname{coker}\Phi,
\end{equation*}
with the same determinant $\lambda$.  Thus
\begin{equation}\label{eq:Qsequence}
0\to \QQ\to \TTzero\to \TT\to0.
\end{equation}
Let $Z_G$ be the zero-dimensional base scheme defined by the entries of $G$, and set
\begin{equation*}
\beta:=\length(E/E_G).
\end{equation*}

\begin{lemma}\label{lem:beta}
One has
\[
\beta=\length(Z_G)=\length(Z_F),
\]
where $Z_F$ is the base scheme of the reduced lift $F$.
\end{lemma}

\begin{proof}
Work in a two-dimensional regular local ring $R$ at a base point. And let $I:=(g_0,\dots,g_4)\subset R$.  For
\[
M:=\operatorname{coker}(R\xrightarrow{(g_0,\dots,g_4)}R^5),
\]
dualizing the exact sequence gives $0\to M^*\to R^5\to I\to0$ and identifies $M^*$ with the kernel in it.  Since the reflexive hull of a height-two ideal in a regular surface is $R$, the reflexive hull $I^{**}$ of $I$ is $R$, equivalently $I^{*}=R$, and a second dual gives
\[
0\to R
\to R^5
\to M^{**}
\to \operatorname{Ext}^1_R(I,R)
\to 0.
\]
The second arrow is the original map $(g_0,\ldots,g_4)$ whose cokernel is $M$. We thus obtain an exact sequence
\[
0\to M\to M^{**}\to\operatorname{Ext}^1_R(I,R)\to0,
\]
which gives $M^{**}/M\simeq \operatorname{Ext}^1_R(I,R)$. From $0\to I\to R\to R/I\to0$, the long exact sequence in $\operatorname{Hom}_R(-,R)$ gives
$\operatorname{Ext}^1_R(I,R)\simeq\operatorname{Ext}^2_R(R/I,R)$,
and local duality shows that $\operatorname{Ext}^2_R(R/I,R)$ has the same length as $R/I$. Thus 
\[
\operatorname{length}(M^{**}/M)=
\operatorname{length}(\operatorname{Ext}^1_R(I,R)) = 
\operatorname{length}(\operatorname{Ext}^2_R(R/I,R)) = 
\operatorname{length}(R/I).
\]
Summing over the base points gives $\length(E/E_G)=\length(Z_G)$.  Coefficient conjugation and $J_{1,4}$ identify the homogeneous ideals of $F$ and $G$, preserving length.
\end{proof}

Fix a component $C=C_\alpha$, choose a general smooth point $z\in C$, let $t=0$ be a local equation of $C$, and put
\[
\Gamma=\pr_1^{-1}(z)\simeq\PP^1.
\]
Since $\lambda=u t^m$ with $m=m_\alpha$ and $u$ a unit, the adjugate identity gives $t^m\TT=0$ near $\Gamma$.  Define
\begin{equation*}
\EE_{j,z}=\left.\left(t^j\TT/t^{j+1}\TT\right)\right|_{\Gamma},
\qquad j\ge0,
\end{equation*}
and use the additive degree convention
\[
\deg\mathcal F=\deg(\mathcal F/\tors\mathcal F)+\length(\tors\mathcal F).
\]
Set
\begin{equation*}
\delta_C:=\sum_{j\ge0}\deg \EE_{j,z}.
\end{equation*}
Note that $\delta_C$ is generically constant. Smith normal form over the transverse DVR gives
\begin{equation}\label{eq:sumranks}
\sum_j\rank \EE_{j,z}=m.
\end{equation}

\begin{lemma}\label{lem:universal}
For every component $C$,
\begin{equation}\label{eq:2m}
\delta_C\ge2m.
\end{equation}
If one layer has generic rank two, then $\delta_C\ge2m+1.$
\end{lemma}
\begin{proof}
At a general point of $C$, the restriction $[F]|_C$ is nonconstant by Lemma~\ref{lem:linear-full} and the incidence identity, so $F(z)\ne0$ and $d[F]_z\ne0$.  In the splitting
\begin{equation*}
\JJ|_\Gamma\simeq\OO\oplus\OO(1)\oplus\OO(2)\oplus\OO(3),
\end{equation*}
the value and first-derivative rows have generic rank two.  Therefore, after removing point torsion, the special cokernel $\EE_0$ is generated by the last two summands $\OO(2)\oplus\OO(3)$.  Multiplication by $t^j$ gives a surjection $\EE_0\twoheadrightarrow \EE_j$. Then there is a generically surjective composition map
\[
\OO(2)\oplus\OO(3) \twoheadrightarrow \EE_0 \twoheadrightarrow \EE_j \twoheadrightarrow\LL_j,
\]
where $\LL_j:=\EE_j/\tors(\EE_j)$ with rank $r_j\le2$.

If $r_j=1$, then surjectivity implies $\deg \LL_j\ge\min\{2,3\}\ge2$; if $r_j=2$, take the determinants, then surjectivity gives $\deg \LL_j\ge5=2r_j+1$.  Hence
\[
\deg \EE_j\ge2r_j+\begin{cases}1,&r_j=2,\\0,&r_j\le1.\end{cases}
\]
Summing and using \eqref{eq:sumranks} proves \eqref{eq:2m}.
\end{proof}

We have the following exact global accounting.
\begin{lemma}\label{lem:account}
One has
\begin{equation}\label{eq:account}
33=\beta+\sum_\alpha e_\alpha\delta_{C_\alpha}.
\end{equation}
In particular,
\begin{equation}\label{eq:budget}
\beta\le9.
\end{equation}
\end{lemma}

\begin{proof}
By generic flatness, after removing a proper closed subset from each
$C_\alpha$, we may assume that the finitely many graded sheaves $t^j\TT/t^{j+1}\TT$ are flat over $C_\alpha$; hence their ranks and degrees on the fibers
$\Gamma$ are constant there.  

Choose a general line
$\ell\subset\PP_Z^2$ meeting every $C_\alpha$ transversely in this flat
locus, away from the other components, and put $
S:=\pr_1^{-1}(\ell).$ We also choose $\ell$ so that its dual point does not belong to $Z_G$.

Since $\TT$ has the locally free resolution
$$
0\to\mathcal A\xrightarrow{\Phi}\JJ\to\TT\to0
$$
and $\det\Phi\ne0$, it is Cohen--Macaulay of pure codimension one.
Thus a general $S$ is $\TT$-regular, and restriction of
\eqref{eq:Qsequence} remains exact.  Moreover $S$ represents the class
$h$, so
$$
33=\int_\II h\,\ch_2(\TTzero)
   =\int_S\ch_2(\TTzero|_S).
$$

The morphism $\pr_2|_S$ is the blow-up of $\PP_W^2$ at the point dual to
$\ell$, hence is an isomorphism near $Z_G$.  Therefore
$\length(\QQ|_S)=\length(E/E_G)=\beta,$ and additivity gives $$33=\beta+\int_S\ch_2(\TT|_S).$$

For $z\in\ell\cap C_\alpha$, transversality makes a local equation $t$ of
$C_\alpha$ also a local equation of the fiber
$\Gamma\subset S$.  Since $t^{m_\alpha}\TT|_S=0$, he associated graded pieces of the $t$-adic filtration are $i_{z*}\EE_{j,z}$ where $i_z:\Gamma\hookrightarrow S$ is the inclusion.

For any coherent sheaf $\mathscr E$ on $\Gamma\simeq\PP^1$,
Grothendieck--Riemann--Roch \cite[\S15.2]{Fulton} gives
$$
\int_S\ch_2(i_{z*}\mathscr E)
=\deg\mathscr E-
\frac{\rank\mathscr E}{2}\deg N_{\Gamma/S}
=\deg\mathscr E,
$$
because $\Gamma$ is a fiber of $S\to\ell$ and hence
$\Gamma^2=0$.  Additivity along the $t$-adic filtration yields
$$
\int_S\ch_2(\TT|_S)
=\sum_\alpha\sum_{z\in\ell\cap C_\alpha}
  \sum_{j\ge0}\deg\EE_{j,z}
=\sum_\alpha e_\alpha\delta_{C_\alpha},
$$
which proves \eqref{eq:account}.  Finally,
Lemma~\ref{lem:universal} and
$\sum_\alpha e_\alpha m_\alpha=\deg\lambda=12$ give
$$
33\ge\beta+2\sum_\alpha e_\alpha m_\alpha
=\beta+24,
$$
and hence $\beta\le9$.
\end{proof}

The remaining task is the following sharper local estimate.
\begin{proposition}\label{prop:jet-defect}
For every irreducible component $C$ of $(\lambda)$,
\begin{equation}\label{eq:prop31}
\delta_C\ge3\ord_C\lambda.
\end{equation}
\end{proposition}

\begin{proof}[Proof of Theorem~\ref{thm:degree-six}]
By Meylan's degree estimate \cite[Theorem 1.4]{Meylan2006},
every rational proper holomorphic map
$f:\B^2\to\B^4$ satisfies $\deg f\le 6$.

Suppose that $\deg f=6$. By \eqref{eq:account} and \eqref{eq:prop31},
\[
33\ge3\sum_C(\deg C)\ord_C\lambda=36,
\]
a contradiction. Hence $\deg f\le 5$.
\end{proof}

\section{Proof of Proposition~\ref{prop:jet-defect}}\label{sec:local-proof}

Fix an irreducible component $C\subset(\lambda)$ and put $m:=\ord_C\lambda.$ We retain the notation of Section~\ref{subsec:prop31}.  Choose a general smooth point
$z\in C$ lying on no other component of $(\lambda)$, let $t$ be a local
equation of $C$ near $z$, and set
\[
\Gamma:=\pr_1^{-1}(z)\simeq\PP^1.
\]
For $W\in\Gamma$, the fiber $L_W:=\pr_2^{-1}(W)$ is the source line through
$z$, and $W\mapsto T_zL_W$ identifies $\Gamma$ with
$\PP(T_z\PP^2)$.  Thus, in affine coordinates centered at $z$, a point of
$\Gamma$ is represented by a direction $v=(s,u)$, up to scale.  Since
$\Phi:\mathcal A\to\JJ$ is the relative third-jet map along the $\pr_2$-fibers,
its restriction to $\Gamma$ is locally represented by
\begin{equation}\label{eq:Phi-directional-representation}
F(z),\qquad \partial_vF(z),\qquad \partial_v^2F(z),\qquad \partial_v^3F(z).
\end{equation}
Indeed, $G(W)$ is constant along $L_W$ and $F^tG(W)=0$ there, so these
relative derivatives descend through the defining quotient of $\AG$; the
extension from $\AG$ to $\mathcal A$ is the morphism $\Phi$ of
Section~\ref{subsec:prop31}.  Thus \eqref{eq:Phi-directional-representation} is a
local matrix for $\Phi|_\Gamma$, not a new map.

Throughout this section, $\partial_v\partial_wF$ denotes the mixed directional derivative
and $\partial_v^jF$ the $j$-fold derivative in $v$; $d[F]_z$ is reserved for the
differential of the projective map $[F]$.

\subsection{Preliminary lemmas}\label{subsec:propagation}
\begin{lemma}\label{lem:gauge-invariance}
Let $a$ be a nowhere-vanishing holomorphic function near $z$ and put
$\widehat F:=aF$.  Then the relative jet map is obtained from $\Phi$ by a
bundle automorphism of $\JJ$.  In particular, the rank of $\Phi|_\Gamma$,
saturated image degrees, cokernel, and the $t$-adic layers are unchanged.
\end{lemma}
\begin{proof}
Leibniz' rule gives
\[
\begin{pmatrix}
\widehat F\\\partial_v\widehat F\\\partial_v^2\widehat F\\\partial_v^3\widehat F
\end{pmatrix}
=
\begin{pmatrix}
a&0&0&0\\
\partial_va&a&0&0\\
\partial_v^2a&2\partial_va&a&0\\
\partial_v^3a&3\partial_v^2a&3\partial_va&a
\end{pmatrix}
\begin{pmatrix}
F\\\partial_vF\\\partial_v^2F\\\partial_v^3F
\end{pmatrix}.
\]
The lower-triangular matrix is invertible because $a$ is a unit and is the
induced change of frame on $\JJ$.  Moreover $(aF)^tG=a(F^tG)$, so the domain
quotient is unchanged.  Hence the two jet maps have isomorphic cokernels,
compatibly with multiplication by $t$.
\end{proof}
Let \(d[F]_z:T_z\PP^2\to T_{[F(z)]}\PP^4\) denote the differential of the projective map \[ [F]:\{F\ne0\}\subset\PP^2\longrightarrow\PP^4.\]

Throughout this subsection, when we speak of the generic rank of \(d[F]\) along \(C\), we mean the rank of this ambient differential at a general point \(z\in C\cap\{F\ne0\}\); in particular, its domain remains \(T_z\PP^2\), rather than \(T_zC\). We distinguish the three possible generic ranks \(0,1,2\).
\begin{lemma}[Boundary projective rank]\label{lem:boundary-rank}
For every $p\in\partial\B^2$, the projective differential $d[F]_p$
has rank two. More precisely, for every nonzero $(1,0)$ tangent vector
$X\in T_p^{1,0}\partial\B^2$, there is a $(1,0)$ transverse
direction $n$, i.e. $n\in T_p^{1,0}\PP^2\backslash T_p^{1,0}\partial \B^2$, such that
$$
\langle F,\partial_{X}F\rangle_{1,4}=0,\qquad
\langle \partial_{X}F,\partial_{X}F\rangle_{1,4}<0,\qquad
\langle F,\partial_{n}F\rangle_{1,4}\ne0.
$$
\end{lemma}

\begin{proof}
By Lemma~\ref{lem:no-base} and the boundary regularity of rational proper
maps \cite{CimaSuffridge1990}, $F$ is a nonvanishing holomorphic lift near
$p$. With the standard source and target sphere defining functions $r$ and
$\rho$, one has locally
$$
\rho\circ f=a\,r,\qquad a(p)>0,
$$
where the positivity follows from the Hopf lemma; see, e.g.,
\cite{DAngeloBook}.

Let $X\in T_p^{1,0}\partial\B^2$ be nonzero. Differentiating the boundary
identity first in the $X$ direction and then in the $\overline X$ direction
at $p$ gives
$$
\langle F,\partial_XF\rangle_{1,4}=0,\qquad
\langle \partial_XF,\partial_XF\rangle_{1,4}
=
a(p)\,(\partial\bar\partial r)_p(X,\overline X)<0.
$$

Choose a $(1,0)$ transverse direction
$n\in T_p^{1,0}\PP^2\setminus T_p^{1,0}\partial\B^2$.
Equivalently, $\partial_n r(p)\ne0$. Differentiating the same identity in
the $n$ direction gives
$$
\langle F,\partial_nF\rangle_{1,4}\ne0.
$$
Hence the classes of $\partial_XF$ and $\partial_nF$ are independent modulo
$\CC F(p)$, and therefore $\rank d[F]_p=2.$
\end{proof}

\begin{lemma}\label{lem:propagation}
Assume that $\Phi|_\Gamma$ has generic corank one.  Let
\[
\LL_j:=\EE_j/\tors(\EE_j).
\]
Then $\LL_j$ is a line bundle for $0\le j<m$, and multiplication by $t^j$
induces a nonzero morphism $\LL_0\to\LL_j$.  Consequently
\begin{equation}\label{eq:propagation-bound}
\deg\LL_j\ge\deg\LL_0,
\qquad
\delta_C\ge m\deg\LL_0.
\end{equation}
\end{lemma}
\begin{proof}
Let $K:=\CC(\Gamma)$ and let $\eta$ be the generic point of $\Gamma$.  Restrict
to a smooth curve through $z$ transverse to $C$ and complete in the transverse
parameter.  Since $\pr_1$ is a $\PP^1$-bundle, the jet map at $\eta$ is a
square matrix over the DVR $K[[t]]$.  Completion preserves the determinant
valuation and the generic ranks of the layers.  The reduction modulo $t$ has
corank one and $\ord_t\det\Phi=m$, so Smith normal form gives
\[
\diag(1,1,1,t^m).
\]
Hence the completed cokernel is $K[[t]]/(t^m)$ and every $\EE_j$, $0\le j<m$,
has generic rank one.  Multiplication by $t^j$ is nonzero on the generic fiber
and descends to a nonzero map $\LL_0\to\LL_j$.  Thus
$\deg\LL_j\ge\deg\LL_0$ on $\PP^1$; summing and adding the nonnegative torsion
lengths proves \eqref{eq:propagation-bound}.
\end{proof}

\subsection{The corank-one case}\label{subsec:corank-one}
Assume throughout this subsection that $\Phi|_\Gamma$ has generic corank one,
equivalently generic rank three.  By Lemma~\ref{lem:propagation}, it is enough
to prove
\begin{equation}\label{eq:L0goal}
\deg\LL_0\ge3.
\end{equation}

\begin{lemma}\label{lem:no-contracted-curve}
There is no irreducible curve \(B\subset\PP^2\) on which the projective map
\([F]\) is constant.
\end{lemma}

\begin{proof}
Suppose that \([F]\) is constant on a dense open subset of \(B\).  Then $F(Z)=\chi_B(Z)a$ there, for a fixed nonzero vector \(a\in\CC^5\) and a scalar function
\(\chi_B=\chi_B(Z)\) on that dense open subset of \(B\).  If \(Z\in B\) and \((Z,W)_{1,2}=0\), the polarized identity
\eqref{eq:polarized-identity} gives \(a^tG(W)=0\) whenever
\(\chi_B(Z)\neq0\).  For general \(W\), the polar line \(L_W\) meets the
dense open subset of \(B\) on which \(\chi_B\neq0\).  Hence
\(a^tG(W)=0\) for general \(W\), and therefore identically.  This gives a
nontrivial linear relation among the entries of \(G\), and hence among
those of \(F\), contradicting Lemma~\ref{lem:linear-full}.
\end{proof}

\subsubsection{Projective differential of rank two}
Let $z\in C$ be general.  Choose affine source coordinates centered at $z$ and
write $v=(s,u)$.  Modulo $\CC F(z)$, choose target vectors $e_1,e_2$ whose
classes form a basis of the image of $d[F]_z$.  Thus
\[
\partial_vF\equiv se_1+ue_2\pmod{\CC F}.
\]
Extend \(F(z),e_1,e_2\) to a fixed basis of the target vector space \(\mathbb C^5\). Modulo $\CC F(z)$, write
\[
\partial_v^2F=(A_2,B_2,Q_{2,1},Q_{2,2}),\qquad
\partial_v^3F=(A_3,B_3,Q_{3,1},Q_{3,2}),
\]
where $A_2,B_2,Q_{2,1},Q_{2,2}$ are scalar binary quadratics in the direction
variables $(s,u)$, and $A_3,B_3,Q_{3,1},Q_{3,2}$ are scalar binary cubics in
$(s,u)$.  After removing the value direction and the two independent
first-jet directions, the residual block of $\Phi|_\Gamma$ is
\begin{equation}\label{eq:residual-rank2}
\begin{pmatrix}
-uA_2+sB_2 & Q_{2,1}&Q_{2,2}\\
-uA_3+sB_3 & Q_{3,1}&Q_{3,2}
\end{pmatrix}
:\OO(-1)\oplus\OO^{\oplus2}\to\OO(2)\oplus\OO(3).
\end{equation}
Indeed, after removing the value row, the first-jet row is $(s,u)$ and its
kernel is generated by the Koszul relation $(-u,s)$.  Applying this relation
to the second- and third-jet rows gives \eqref{eq:residual-rank2}.

Since $\Phi|_\Gamma$ has generic rank three, \eqref{eq:residual-rank2} has
generic rank one.  Let $M\simeq\OO(d)$ be its saturated image line.  Then
\[
\deg\LL_0
 \ge \deg\bigl(\OO(2)\oplus\OO(3)\bigr)-\deg M
 =5-\deg M.
\]
Thus \eqref{eq:L0goal} holds unless $d=3$.

Assume $M\simeq\OO(3)$.  Since $\operatorname{Hom}(\OO(3),\OO(2))=0$, the $\OO(2)$-projection of $M$ vanishes, so
\[
Q_{2,1}=Q_{2,2}=0,
\qquad
-uA_2+sB_2=0.
\]
Hence there is a scalar linear form $\varrho_0(s,u)$ in the direction variables, varying holomorphically with the base point of $C$, with
\[
A_2=s\varrho_0,
\qquad
B_2=u\varrho_0.
\]
Putting $\varrho=\varrho_0/2$ and polarizing gives the pure-trace identity
\begin{equation}\label{eq:pure-trace}
\partial_v\partial_wF(z)\equiv
\varrho(v)\partial_wF(z)+\varrho(w)\partial_vF(z)
\pmod{\CC F(z)}
\end{equation}
for all tangent vectors $v,w$ at general points of $C$.

\begin{lemma}\label{lem:pure-trace-reduction}
After shrinking around a general point of \(C\), there is a nowhere-zero
holomorphic function \(a\) such that \(\widehat F=aF\) satisfies
\begin{equation}\label{eq:hessian-scalar}
 \partial_X\partial_Y\widehat F\in\CC\widehat F
 \qquad\text{on }C
\end{equation}
for all local vector fields \(X,Y\).  Consequently the subspaces
\[
 V_z:=\Span\{\widehat F(z),d\widehat F_z(T_z\PP^2)\}
\]
are constant along \(C\).  If \(\tau\) is a nonzero tangent field along
\(C\), then
\[
 U_z:=\Span\{\widehat F(z),\partial_\tau\widehat F(z)\}
\]
is also constant on a dense open subset of \(C\).  Writing the resulting
fixed spaces as \(U\subset V\subset\CC^5\), and writing \(c=0\) for the
irreducible equation of \(C\), one has
\begin{equation}\label{eq:c3-div}
 c^3\mid n(F)
 \qquad\text{for every }n\in\operatorname{Ann}(V),
\end{equation}
where $\operatorname{Ann}(V):=\{n\in(\CC^5)^\vee:n|_V=0\}.$ In particular, \(\deg C=1\).
\end{lemma}

\begin{proof}
Choose local coordinates \((\xi,t)\) with \(C=(t=0)\).  Note that on this coordinate, there exists a local holomorphic function $\psi$ such that
\[
 \partial_X\psi=-\varrho(X)\qquad\text{on }C
\]
for every local vector field \(X\). Set \(a=e^\psi\) and
\(\widehat F=aF\).  By Leibniz' rule and \eqref{eq:pure-trace},
\[
\begin{aligned}
\partial_X\partial_Y\widehat F
&=a\partial_X\partial_YF+(\partial_Xa)\partial_YF+(\partial_Ya)\partial_XF+(\partial_X\partial_Ya)F\\
&\equiv
\bigl(a\varrho(X)+\partial_Xa\bigr)\partial_YF
+\bigl(a\varrho(Y)+\partial_Ya\bigr)\partial_XF
\pmod{\CC F}.
\end{aligned}
\]
Since \(\partial_Xa/a=\partial_X\psi=-\varrho(X)\) on \(C\), the two first-derivative
terms cancel and \eqref{eq:hessian-scalar} follows.  By
Lemma~\ref{lem:gauge-invariance}, this unit gauge does not change any of the
jet-cokernel data.

Let \(\tau\) be tangent to \(C\).  Differentiating the generators
\(\widehat F\) and \(\partial_X\widehat F\) in the \(\tau\)-direction produces
vectors that remain in \(V_z\) by \eqref{eq:hessian-scalar}.  Hence \(V_z\) is constant.  Since \([F]|_C\) is nonconstant by
Lemma~\ref{lem:no-contracted-curve}, \(U_z\) has dimension two at a general
point.  Applying \eqref{eq:hessian-scalar} with \(X=Y=\tau\) shows in the
same way that \(U_z\) is constant.

If \(n\in\operatorname{Ann}(V)\), then \(n(\widehat F)\) and all of its derivatives of
order at most two vanish along \(C\).  Since \(a\) is a unit, the same
order of vanishing holds for \(n(F)\), proving \eqref{eq:c3-div}.

Finally, \(\dim V\le3\), hence \(\dim\operatorname{Ann}(V)\ge2\).  If
\(\deg C\ge3\), then \(\deg c^3>6\), so \eqref{eq:c3-div} forces
\(n(F)\equiv0\) for every \(n\in\operatorname{Ann}(V)\), contradicting linear fullness.
If \(\deg C=2\), each \(n(F)\) is a scalar multiple of the sextic \(c^3\).
Two independent functionals in \(\operatorname{Ann}(V)\) therefore have a nontrivial
linear combination that annihilates \(F\) identically, again contradicting
linear fullness.  Thus \(\deg C=1\).
\end{proof}

By Proposition~\ref{prop:A-boundary-line-exclusion}, the exceptional line \(C\) does not meet the source sphere. Let \(\Lambda_C\subset \mathbb C^3\) be the two-dimensional subspace such that \(C=\mathbb P(\Lambda_C)\). Then the source Hermitian form has no nonzero null vector on \(\Lambda_C\), and hence its restriction to \(\Lambda_C\) is definite. Since the ambient Hermitian form has positive index one, this restriction cannot be positive definite. Therefore \(\Lambda_C\) is negative definite, and \(C\) is Hermitian negative.

We now work on the negative line.  Choose unitary source coordinates
\begin{equation}\label{eq:negcoords}
Z=(x,z_1,z_2),\qquad W=(y,w_1,w_2),\qquad C=(x=0),
\end{equation}
and write $z=(z_1,z_2)$ and $w=(w_1,w_2)$ for the binary source and polar variables.  Thus
\begin{equation}\label{eq:sigma}
(Z,W)_{1,2}=xy-\sigma(z,w),
\qquad
\sigma:=z_1w_1+z_2w_2.
\end{equation}
After factoring out the common divisor of the components of \(F|_C\), we may write
\begin{equation}\label{eq:F-on-C}
 F|_C=\chi\,Y,\qquad
 Y=(Y_1,Y_2)\in U,\qquad
 \gcd(Y_1,Y_2)=1,
\end{equation}
where \(\chi=\chi(z)\) is a homogeneous scalar binary form and \(Y=Y(z)\) is a homogeneous \(U\)-valued binary form in \(z=(z_1,z_2)\), with $\nu:=\deg Y$ and $\deg\chi=6-\nu.$

The map \([Y]:C\simeq\PP^1\to\PP(U)\) is nonconstant by
Lemma~\ref{lem:no-contracted-curve}.  Restricting the complexified incidence
identity to \(C\) gives
\[
 (Y(z),Y^*(w))_{1,4}=0
 \qquad\text{whenever }\sigma(z,w)=0.
\]
If \(u\in C\) is the Hermitian polar point of \(z\), then
\(w=\overline u\) satisfies \(\sigma(z,w)=0\) and
\(Y^*(w)=\overline{Y(u)}\).  Hence the preceding complexified relation is
equivalent to
\[
 \langle Y(z),Y(u)\rangle_{1,4}=0.
\]

The Hermitian form \(\langle\, ,\,\rangle_{1,4}\) restricts nondegenerately to \(U\). Indeed, otherwise its radical would be a line \(R\subset U\). For a general \(u\in C\) with \(Y(u)\notin R\), the orthogonal complement of \(Y(u)\) in \(U\) is exactly \(R\). If \(z=\iota_C(u)\), the incidence relation gives $
\langle Y(z),Y(u)\rangle_{1,4}=0,$ hence \(Y(z)\in R\). Since \(\iota_C\) is bijective, this would force \([Y](C)\subset \mathbb P(R)\), contradicting the nonconstancy of \([Y]\).

Choose \(e_\perp\in V\cap U^\perp\) spanning \(V/U\), and let $ [\partial_xF]_{V/U}$ be $(\partial_xF\bmod U)\in V/U$ . Hence
\[
 [\partial_xF]_{V/U}=g_C\,e_\perp,
\]
where \(g_C=g_C(z)\) is a binary quintic on \(C\). Choose a local affine coordinate \(\xi\) on \(C\) near the general point under consideration, and set $\tau:=\frac{\partial}{\partial \xi}.$ For any local function \(\alpha\) on \(C\), we write $\alpha':=\frac{\partial \alpha}{\partial \xi}.$ With this convention, define the binary Wronskian of \(Y=(Y_1,Y_2)\) by
$$
\operatorname{Wr}(Y):=Y_1Y_2'-Y_2Y_1'.
$$

Taking $v=w=\tau$, and then taking $v=\tau$ with $w$ in the normal direction to $C$, gives the following two expressions:
\[
 2\varrho(\tau)
 =2\frac{\chi'}{\chi}+\frac{\operatorname{Wr}(Y)'}{\operatorname{Wr}(Y)},
 \qquad
 \frac{g_C'}{g_C}=\varrho(\tau).
\]
Therefore $g_C^2=\gamma\,\chi^2\operatorname{Wr}(Y)$ for some \(\gamma\in\CC^\times\).  Since the polynomial ring in two
variables is a UFD, \(\chi\mid g_C\).  Write
\[
 g_C=\chi\rho,\qquad \deg\rho=\nu-1,
\]
where \(\rho=\rho(z)\) is a binary form on \(C\).
Then
\begin{equation}\label{eq:Wronskian-square}
 \operatorname{Wr}(Y)=\gamma^{-1}\rho^2.
\end{equation}

Let $
C=\PP(\Lambda_C),$ where \(\Lambda_C\subset\CC^3\) is a negative-definite two-plane.  Define
$$
\iota_C([z])
:=\PP\!\left(\Lambda_C\cap z^{\perp_{1,2}}\right),
\qquad
z^{\perp_{1,2}}
=\{w\in\CC^3:\langle z,w\rangle_{1,2}=0\}.
$$
Similarly, since \(U\subset\CC^5\) is nondegenerate, define
$$
\iota_U([u])
:=\PP\!\left(U\cap u^{\perp_{1,4}}\right),
\qquad
u^{\perp_{1,4}}
=\{v\in\CC^5:\langle u,v\rangle_{1,4}=0\}.
$$
The incidence identity \eqref{eq:FG-zero-on-I} on \(C\), interpreted on
the Hermitian real form as above, gives
\begin{equation}\label{eq:polarity-Y}
 [Y]\circ\iota_C=\iota_U\circ[Y].
\end{equation}

\begin{lemma}\label{lem:degree-alternative}
With the notation above, $\nu\in\{1,3,5\}.$
\end{lemma}

\begin{proof}
By \eqref{eq:Wronskian-square}, every ramification multiplicity of
\([Y]\) is even.  By \eqref{eq:polarity-Y}, the ramification divisor is
\(\iota_C\)-invariant.  Because \(\iota_C\) has no fixed point, its points
occur in pairs; hence the total ramification degree is divisible by four.
Riemann--Hurwitz gives $\deg\operatorname{Ram}([Y])=2\nu-2,$ so \(\nu\) is odd.  Since \(Y_1,Y_2\) are homogeneous components of the
degree-six map \(F|_C\), one has \(1\le\nu\le6\).  Therefore
\(\nu\in\{1,3,5\}\). 
\end{proof}
Since \(U\) is nondegenerate and the ambient Hermitian form has positive index one, its restriction to \(U\) has signature either \((0,2)\) or \((1,1)\). Suppose it had signature \((1,1)\). Choose a fixed point \(q\) of the polarity \(\iota_U\) outside the finite branch-value set of \([Y]\). Then \([Y]^{-1}(q)\) consists of \(\nu\) distinct points and, by \eqref{eq:polarity-Y}, is invariant under the fixed-point-free involution \(\iota_C\). Hence its cardinality is even, contradicting \(\nu\in\{1,3,5\}\). Therefore the Hermitian form on \(U\) is negative definite.

The following normal form will be used in all three cases.

\begin{lemma}[Projection normal form]\label{lem:projection}
There exist scalar homogeneous polynomials $\widehat\chi(Z)$ and
$\widetilde g(Z)$ in $Z=(x,z_1,z_2)$ of degrees $6-\nu$ and $5$,
respectively, and a $U^\perp$-valued homogeneous cubic polynomial $K(Z)$ in
$Z$ such that
\begin{equation}\label{eq:projectionform}
F=\widehat\chi Y+x\widetilde g\,e_\perp+x^3K,
\end{equation}
with
\begin{equation}\label{eq:projectionboundary}
\widehat\chi(0,z)=\chi(z),
\qquad
\widetilde g(0,z)=\chi(z)\rho(z).
\end{equation}
\end{lemma}
\begin{proof}
Let $\pi_U$ be the Hermitian projection onto $U$.  For general $Z=(x,z)$, choose the polar point $W_0=(0,w)\in C$ with $\sigma(z,w)=0$.  The incidence identity gives
\[
(\pi_UF(Z),Y^*(w))_{1,4}=0.
\]
Inside the nondegenerate two-plane $U$, the kernel of this functional is the line $\CC Y(z)$.  Coprimality of $Y_1,Y_2$ gives $\pi_UF=\widehat\chi Y$.  The difference $F-\widehat\chi Y$ takes values in $U^\perp$ and is divisible by $x$, while its image modulo $V$ is divisible by $x^3$ by \eqref{eq:c3-div}.  Splitting $U^\perp=\CC e_\perp\oplus E_N$ for a fixed complementary 
subspace $E_N$ gives \eqref{eq:projectionform}, and the transverse first derivative gives \eqref{eq:projectionboundary}.
\end{proof}

We now treat the three alternatives in Lemma~\ref{lem:degree-alternative}.

\textbf{Case $\nu=1$.}
Here $\rho$ is constant and $\deg\widehat\chi=5$.  From \eqref{eq:projectionboundary}, we have $\widetilde g-\rho\widehat\chi=xa$ for a scalar homogeneous quartic $a(Z)$ in $Z=(x,z_1,z_2)$, so
\begin{equation*}
F=\widehat\chi(Y+x\rho e_\perp)+x^2(ae_\perp+xK).
\end{equation*}
Hence every coordinate of $F$ lies in $(\widehat\chi,x^2)$.  Since $x$ and $\widehat\chi$ have no common component, this complete intersection has length $5\cdot2=10$ and is contained in the base scheme $Z_F$.  Thus by Lemma~\ref{lem:beta}, we have $\beta\ge10$, contradicting \eqref{eq:budget}.

\textbf{Case $\nu=3$.}
The ramification divisor consists of two double points exchanged by polarity.  After unitary changes on the negative source and target two-planes,
\begin{equation*}
Y=(z_1^3,z_2^3),
\qquad
\rho=z_1z_2,
\qquad
(Y,Y^*)_{1,4}=-(z_1^3w_1^3+z_2^3w_2^3).
\end{equation*}
On a connected open set where $\widehat\chi\ne0$, divide the lift by
$\widehat\chi$ and expand only the $e_\perp$-coefficient:
\[
F=Y+x\rho e_\perp+x^2q e_\perp+x^3K.
\]
Here $q=q(x,z)$ is a scalar holomorphic function on this open set, while, after
renaming $K/\widehat\chi$ by $K$, the symbol $K=K(x,z)$ denotes a holomorphic
$U^\perp$-valued map there.  This local $K$ is the normalized version of the
homogeneous cubic $K(Z)$ in Lemma~\ref{lem:projection}.
Let $\varepsilon:=\langle e_\perp,e_\perp\rangle_{1,4}$.  Using $xy=\sigma$ and
\[
-(Y,Y^*)_{1,4}=z_1^3w_1^3+z_2^3w_2^3=\sigma^3-3z_1z_2w_1w_2\sigma,
\]
the terms of lowest total order in $(x,y)$ are $(3+\varepsilon)xy\,\rho\rho^*.$ Hence $\varepsilon=-3$.  The next order is
$\varepsilon xy(xq\rho^*+y\rho q^*)$; comparison of its $x^2y$ and $xy^2$ parts gives $q=0$.  Restoring the homogeneous lift,
\begin{equation*}
F=\widehat\chi(Y+x\rho e_\perp)+x^3R_3,
\end{equation*}
where $R_3(Z)$ is a $\CC^5$-valued homogeneous cubic polynomial in $Z$.
Thus $(F_0,\dots,F_4)\subset(\widehat\chi,x^3)$ and $\deg\widehat\chi=3$, so $\beta\ge9$.  By \eqref{eq:budget}, $\beta=9$ and the complete intersection exhausts the base scheme.  Hence scheme-theoretically
$Z_F=(\widehat\chi,x^3)$ and, by conjugation, $Z_G=(\widehat\chi^*,y^3)$.

Let $W_C=[1:0:0]$ be the polar point corresponding to the line $C=(x=0)$, so $L_{W_C}=C$.  Since $y(W_C)\ne0$, the preceding description gives $W_C\notin Z_G$; therefore the reflexive extension $\mathcal A$ agrees with the original quotient $\mathcal A_G$ near $W_C$.  Since $F|_C=\chi Y$ takes values in the fixed two-plane $U$, all four relative jets at $W_C$ lie in $U$, and the special jet rank drops from its generic value three to at most two.  Thus the saturated rank-three image on $\Gamma$ is strictly larger than the actual image at $W_C$, and the generic rank-one cokernel $\EE_0$ has point torsion of length at least one there.  Consequently $\delta_C\ge2m+1.$ Using the weak estimate for all other components and the exact accounting,
\[
33=9+\sum_D(\deg D)\delta_D
\ge9+2\sum_D(\deg D)m_D+1
=34,
\]
a contradiction. Hence we only need to deal with the case $\nu=5$.

For use below, we make explicit a simple rank notion.  If
\(R(z,w)\) is bihomogeneous, its \emph{separated rank} is the least integer
\(r\) for which, for suitable scalar forms \(a_j=a_j(z)\) and \(b_j=b_j(w)\),
\[
 R(z,w)=\sum_{j=1}^r a_j(z)b_j(w).
\]
A bihomogeneous polynomial \(B(z,w)\) of bidegree \((d,d)\) will be called a \emph{Hermitian kernel} if its coefficient matrix on \(\Sym^d\CC^2\) is Hermitian; equivalently, writing
$$
B(z,w)=\sum_{|\mu|=|\nu|=d} b_{\mu\nu}z^\mu w^\nu,
$$
one has \(b_{\mu\nu}=\overline{b_{\nu\mu}}\).  In particular,
\(B(z,\overline z)\in\mathbb R\).
\begin{lemma}\label{lem:seprank}
Let \(\sigma(z,w)=z_1w_1+z_2w_2\).
\begin{enumerate}[label=\textup{(\roman*)}]
\item The polynomial \(\sigma\) is irreducible.  In particular, it cannot
divide a nonzero separated product \(a(z)b(w)\).
\item If \(0\neq R(z,w)\) is divisible by \(\sigma^q\) for an integer \(q\ge0\), then the separated
rank of \(R\) is at least \(q+1\).
\end{enumerate}
\end{lemma}
\begin{proof}
For (i), note that \(\sigma\) is irreducible in the UFD
\(\CC[z_1,z_2,w_1,w_2]\).  A nonzero product \(a(z)b(w)\) has no mixed
irreducible factor, so it cannot be divisible by \(\sigma\).

For (ii), let \(Q\ge q\) be maximal such that \(\sigma^Q\mid R\).
Choose \(w_0\) so that \(\sigma(z,w_0)\) does not divide
\((R/\sigma^Q)(z,w_0)\), and choose a \(w\)-direction \(\eta\) such that
\(\left.\partial_\eta\sigma(z,w)\right|_{w=w_0}\) is not proportional to
\(\sigma(z,w_0)\).  For \(0\le j\le Q\), consider the polynomials $\left.\partial_\eta^jR(z,w)\right|_{w=w_0}.$ It is elementary that $\sigma(z,w_0)^{Q-j}|\left.\partial_\eta^jR(z,w)\right|_{w=w_0}$ but $\sigma(z,w_0)^{Q+1-j}$ does not divide it.
So they are linearly independent.  Thus
the separated rank is at least \(Q+1\ge q+1\).
\end{proof}

\textbf{Case $\nu=5$.}
Here $\chi(z)$ is linear and, since $\widehat\chi(Z)$ is also linear, we have $\widehat\chi=\chi+\alpha x$ for some constant $\alpha\in\CC$.
Write the projection normal form as
\begin{equation*}
F=\widehat\chi Y+x\widetilde ge_\perp+x^3K,
\qquad
\widetilde g(0,z)=\chi\rho.
\end{equation*}
Since $(Y,Y^*)_{1,4}$ vanishes on $\sigma=0$, write $(Y,Y^*)_{1,4}=\sigma A_4,$ where $A_4=A_4(z,w)$ is a scalar bihomogeneous polynomial of bidegree $(4,4)$ in $(z,w)$.  Put $\varepsilon=\langle e_\perp,e_\perp\rangle_{1,4}$.  Expanding the polarized identity, cancelling one factor $\sigma=xy$, and then setting $x=y=0$ gives
\begin{equation}\label{eq:C4modsigma}
A_4+\varepsilon\rho\rho^*\equiv0\pmod\sigma.
\end{equation}
If $\varepsilon=0$, then $\sigma^2\mid(Y,Y^*)_{1,4}$, but the left side has separated rank at most two while Lemma~\ref{lem:seprank}(ii) gives separated rank at least three.  Thus $\varepsilon\ne0$.

Absorb the $e_\perp$-component of $x^3K$ into $\widetilde g$ and write
\begin{equation}\label{eq:nu5H}
F=\widehat\chi Y+x\widetilde g,e_\perp+x^3H,
\qquad H=H(Z)\in V^\perp,
\end{equation}
where $H(Z)$ is a $V^\perp$-valued homogeneous cubic polynomial in $Z$.
The polarized identity becomes
\begin{equation}\label{eq:nu5basic}
\widehat\chi\widehat\chi^*A_4
+\varepsilon\widetilde g\widetilde g^*
+\sigma^2(H,H^*)_{1,4}=0.
\end{equation}
Set $y=0$, so $\sigma=0$.  Using \eqref{eq:C4modsigma} and $\widetilde g^*(0,w)=\chi^*\rho^*$ gives
$$
\varepsilon\chi^*\rho^*(\widetilde g-\widehat\chi\rho)=0\pmod\sigma.
$$

The quotient by $(\sigma)$ is a domain, and the first factor is nonzero there.  Hence
$\widetilde g-\widehat\chi\rho\in(\sigma)$.  This polynomial depends only on $(x,z)$, while no nonzero polynomial in $(x,z)$ alone is divisible by the irreducible form $\sigma(z,w)$.  Therefore
\begin{equation}\label{eq:gexact4}
\widetilde g=\widehat\chi\rho.
\end{equation}
At the positive source point $p_+=[1:0:0]$, the vector $H(p_+)=F(p_+)$ is positive.  If $\varepsilon>0$, then $V$ has signature $(1,2)$ and $V^\perp$ is negative definite, impossible.  Hence
\begin{equation*}
\varepsilon<0,
\qquad
V\text{ is negative definite},
\qquad
\operatorname{sign}(V^\perp)=(1,1).
\end{equation*}

Set
$C_4(z,w)=A_4(z,w)+\varepsilon\rho(z)\rho^*(w)$, a scalar bihomogeneous
polynomial of bidegree $(4,4)$.  Equation \eqref{eq:nu5basic} becomes $\widehat\chi\widehat\chi^*C_4+\sigma^2(H,H^*)_{1,4}=0.$ After localizing at $x$ and substituting $y=\sigma/x$, the reduction of $\widehat\chi\widehat\chi^*$ modulo $\sigma$ is $\widehat\chi\chi^*\ne0$.  The identity forces $\sigma^2\mid C_4$.  Since $C_4$ lies in the original polynomial ring, the same divisibility holds before localization.  Write
\begin{equation}\label{eq:B2section}
C_4:=-\sigma^2\mathcal B_2,
\qquad
(H,H^*)_{1,4}=\widehat\chi\widehat\chi^*\mathcal B_2,
\end{equation}
where $\mathcal B_2=\mathcal B_2(z,w)$ is a scalar Hermitian bihomogeneous polynomial of
bidegree $(2,2)$ in $(z,w)$.
Multiplying the first identity by $-\sigma$ and using a negative orthonormal basis of $U$ gives
\begin{equation*}
\sigma^3\mathcal B_2
=Y_1Y_1^*+Y_2Y_2^*
-\varepsilon(z_1\rho)(w_1\rho^*)
-\varepsilon(z_2\rho)(w_2\rho^*).
\end{equation*}
Thus the separated rank of $\sigma^3\mathcal B_2$ is at most four, $\mathcal B_2(z,\bar z)>0$, and $\sigma\nmid \mathcal B_2$ by Lemma~\ref{lem:seprank}(ii).

After rotating $x$, assume $\alpha$ is real and $\alpha\ge0$.  Write
\[H(Z)=H_0(z)+xH_1(z)+x^2H_2(z)+x^3u_+,
\qquad u_+:=H(p_+),\]
where $H_0,H_1,H_2$ are $V^\perp$-valued binary forms in $z$ of degrees
$3,2,1$, respectively, and $u_+\in V^\perp$ is a constant target vector.
Choose a negative vector $u_-$ spanning $u_+^\perp$ in the signature $(1,1)$ space $V^\perp$, and set $a_+:=\langle u_+,u_+\rangle>0$ and $a_-:=-\langle u_-,u_-\rangle>0.$

The condition $\sigma\nmid \mathcal B_2$ implies $H_0\ne0$.  Substitute $y=\sigma/x$ in the second identity of \eqref{eq:B2section}.  Its $x^3$ coefficient gives
$\langle u_+,H_0^*\rangle=0$, hence $H_0=u_-P_3$ with
$P_3=P_3(z)$ a nonzero scalar binary cubic.  Write
$$
H_1=u_-P_2+u_+Q_2,\qquad H_2=u_-Q_1+u_+P_1,
$$
where $P_2(z),Q_2(z)$ are scalar binary quadratics and $P_1(z),Q_1(z)$ are
scalar binary linear forms.
The $x^2$ coefficient is
$-a_-Q_1P_3^*+a_+\sigma Q_2^*=0$.  Since $\sigma$ cannot divide the nonzero separated product $Q_1P_3^*$, Lemma~\ref{lem:seprank}(i) gives $Q_1=Q_2=0$.  Thus
\begin{equation*}
H=u_-(P_3+xP_2)+u_+x^2(P_1+x),
\qquad P_3\ne0,
\end{equation*}
where, consistently, $P_j(z)$ is a scalar binary form of degree $j$.  The coefficients of $x$ and $x^0$ in the Hermitian kernel identity are
\begin{align}
-a_-P_2P_3^*+a_+\sigma^2P_1^*&=\alpha\chi^*\mathcal B_2,\label{eq:x1section}\\
-a_-P_3P_3^*-a_-\sigma P_2P_2^*+a_+\sigma^2P_1P_1^*+a_+\sigma^3
&=(\chi\chi^*+\alpha^2\sigma)\mathcal B_2.\label{eq:x0section}
\end{align}
Multiply \eqref{eq:x1section} by $\chi$, subtract $\alpha$ times \eqref{eq:x0section}, and reduce modulo $\sigma$.  The $\mathcal B_2$-terms cancel, giving
$$
(\chi P_2-\alpha P_3)P_3^*\equiv0\pmod\sigma.
$$
Since $P_3^*$ depends only on $w$, the irreducible form $\sigma(z,w)$ cannot divide it.  Thus $\sigma$ divides $\chi P_2-\alpha P_3$; the latter depends only on $z$, so it must vanish.  Hence
\begin{equation}\label{eq:keysection}
\chi P_2=\alpha P_3.
\end{equation}

If $\alpha>0$, then $P_3=(\chi/\alpha)P_2$.  Substitution into \eqref{eq:x1section} shows that $\chi^*$ divides $\sigma^2P_1^*$.  Since $\chi^*\nmid \sigma$, it follows that $\chi^*\mid P_1^*$, so $P_1=\kappa\chi$.  Eliminating $\mathcal B_2$ between \eqref{eq:x1section} and \eqref{eq:x0section} gives
$$
\left(|\kappa|^2-\frac{\bar\kappa}{\alpha}\right)\chi\chi^*
+(1-\alpha\bar\kappa)\sigma=0.
$$

Since $\chi\chi^*$ and $\sigma$ are linearly independent, $\kappa=1/\alpha$.  Hence
\begin{equation}\label{eq:Hdivsection}
H=\frac{\widehat\chi}{\alpha}(u_-P_2+u_+x^2),
\end{equation}
so $\widehat\chi\mid H$.

If $\alpha=0$, \eqref{eq:keysection} gives $P_2=0$, and
\eqref{eq:x1section} gives $P_1=0$. After unitary normalization
$\chi=z_2$, \eqref{eq:x0section} at the zero of $\chi$ gives
$$
P_3=\mathbf c\,z_1^3,
\qquad
a_-|\mathbf c|^2=a_+,
$$
for a nonzero scalar $\mathbf c\in\CC$. One obtains
$$
\mathcal B_2=a_+(3(z_1w_1)^2+3z_1z_2w_1w_2+(z_2w_2)^2),
$$
so $\sigma^3\mathcal B_2$ has diagonal coefficient matrix
$$
a_+\operatorname{diag}(3,12,19,15,6,1)
$$
in the degree-five monomial basis, of rank six. Hence the separated rank of $\sigma^3\mathcal B_2$ is six, contradicting the previously established upper bound of four. Thus $\alpha=0$ is impossible. In the
remaining case \eqref{eq:Hdivsection} and \eqref{eq:nu5H},
\eqref{eq:gexact4} show that every coordinate of $F$ is divisible by
the nonconstant linear form $\widehat\chi$, contradicting reducedness.

All three values in Lemma~\ref{lem:degree-alternative} are impossible.  Hence $M\not\simeq\OO(3)$ and \eqref{eq:L0goal} holds whenever $d[F]$ has rank two.

\subsubsection{Projective differential of rank one}
At a general point choose coordinates with
\[
\partial_vF\equiv se_1\pmod{\CC F}.
\]
After removing the value row, $\Phi|_\Gamma$ has the form
\begin{equation*}
\begin{pmatrix}
s&0&0&0\\
A_2&Q_{2,1}&Q_{2,2}&Q_{2,3}\\
A_3&Q_{3,1}&Q_{3,2}&Q_{3,3}
\end{pmatrix}
:\OO^{\oplus4}\to\OO(1)\oplus\OO(2)\oplus\OO(3).
\end{equation*}
Here $A_2,Q_{2,1},Q_{2,2},Q_{2,3}$ are scalar binary quadratics in $(s,u)$,
and $A_3,Q_{3,1},Q_{3,2},Q_{3,3}$ are scalar binary cubics in $(s,u)$.
The last three columns have saturated image $M\simeq\OO(d)$ in $\OO(2)\oplus\OO(3)$ with $d\le3$.  Modulo $M$, the first column has saturated image $N\simeq\OO(e)$ with $e\le1$.  Hence
\[
\deg \LL_0\ge6-d-e\ge3
\]
unless $(d,e)=(3,1)$.

In the extremal pair, $M\simeq\OO(3)$ has zero $\OO(2)$-projection, so $Q_{2,j}=0$, while $N\to\OO(1)$ is an isomorphism up to scalar and $A_2=s\alpha_1$ for a scalar linear form $\alpha_1(s,u)$.  Polarization gives the same pure-trace identity \eqref{eq:pure-trace}.  Lemma~\ref{lem:pure-trace-reduction} applies with a fixed two-plane $V$, and $C$ is a line with $c^3\mid n(F)$ for $n\in\Ann(V)$.

This line is Hermitian negative. If it met the sphere,
Lemma~\ref{lem:boundary-rank} would give projective differential rank two
there. On the other hand, \eqref{eq:c3-div} implies
$$
F(z),\,\partial_XF(z)\in V
\qquad\text{for every }X\in T_z\PP^2,
$$
and hence$\rank d[F]_z\le \dim\bigl(V/\CC F(z)\bigr)=1,$ a contradiction.

In negative-line coordinates \eqref{eq:negcoords}--\eqref{eq:sigma}, write
$F|_C=\chi(z)Y(z)$, where $\chi(z)$ is a scalar homogeneous binary form and
$Y(z)=(Y_1(z),Y_2(z))$ is a $V$-valued homogeneous binary polynomial map with
coprime scalar coordinates.  The Hermitian form on $V$ is nondegenerate by the
same polarity argument as above.  Orthogonal projection onto $V$ gives
\begin{equation*}
F=\omega Y+x^3H,
\qquad H\in V^\perp,
\end{equation*}
where $\omega=\omega(Z)$ is a scalar homogeneous polynomial of degree
$6-\deg Y$ and $H=H(Z)$ is a $V^\perp$-valued homogeneous cubic polynomial.
Since $(Y,Y^*)_{1,4}=\sigma A$, where $A=A(z,w)$ is a scalar bihomogeneous
polynomial of bidegree $(\deg Y-1,\deg Y-1)$, substitution into the incidence identity gives
\[
\omega\omega^*A+x^2y^2(H,H^*)_{1,4}=0.
\]
Setting $x=y=0$ shows $\sigma\mid A$, hence \begin{equation}\label{eq:corank-one-double-div}\sigma^2\mid(Y,Y^*)_{1,4}.\end{equation}
At a general polar pair $(z,w)$ with $\sigma(z,w)=0$, and for a nonzero tangent direction $\eta$ at $w$, the double divisibility gives
\[
(Y(z),Y^*(w))_{1,4}=0,
\qquad
(Y(z),\partial_\eta Y^*(w))_{1,4}=0.
\]
Since $V$ is a nondegenerate two-plane, the kernel of pairing with the nonzero vector $Y(z)$ is one-dimensional.  Thus $Y^*(w)$ and $\partial_\eta Y^*(w)$ are dependent, so $d[Y^*]=0$ generically and $[Y]$ is constant, contradicting Lemma~\ref{lem:no-contracted-curve}.  The extremal pair is impossible, and therefore $\deg \LL_0\ge3$ in the rank-one branch.

\subsubsection{Projective differential of rank zero}
If $d[F]$ has generic rank zero along $C$, then its tangent restriction vanishes generically, so $[F]|_C$ is constant.  This contradicts Lemma~\ref{lem:no-contracted-curve}.

We have proved $\deg \LL_0\ge3$ in every corank-one configuration.  Lemma~\ref{lem:propagation} therefore gives
\begin{equation}\label{eq:corank-one-complete}
\delta_C\ge3m.
\end{equation}

\subsection{Higher-corank cases}\label{subsec:higher-corank}
Assume $\Phi|_\Gamma$ has generic corank at least two, so its rank is at most two.  The generic rank of $d[F]$ along $C$ cannot be zero by Lemma~\ref{lem:no-contracted-curve}.  If $d[F]$ has positive generic rank, then $F$ and $\partial_vF$ are independent for a general direction, and therefore
$\Phi|_\Gamma$ has generic rank exactly two.  It remains to exclude the cases in which $d[F]$ has generic rank one or two.

\begin{lemma}\label{lem:higher-rank-one-rigidity}
There is no component $C$ on which $\Phi|_\Gamma$ has generic rank two and $d[F]$ has generic rank one.
\end{lemma}
\begin{proof}
Following the argument of Lemma~\ref{lem:pure-trace-reduction}, we first show that
\(\deg C=1\).  For general \(z\in C\), set
\[
E_z:=\Span\{F(z),dF_z(T_z\PP^2)\}.
\]
Since \(d[F]\) has generic rank one along \(C\), \(\dim E_z=2\).
The rank-two assumption on \(\Phi|_\Gamma\) implies, for general \(v\),
\[
\partial_v^2F(z),\ \partial_v^3F(z)\in E_z.
\]
The relevant minors are polynomial in \(v\), so the same inclusion holds for
every direction \(v\).  Thus all second and third
derivatives of \(F\) lie in \(E_z\).  Hence differentiation along \(C\)
preserves \(E_z\), so \(E_z\) is constant on a dense open subset of \(C\);
write it as \(E\subset\CC^5\).

For every \(n\in\Ann(E)\), \(n(F)\) and its derivatives up to order three
vanish along \(C\).  Thus
\[
c^4\mid n(F).
\]
Since \(\dim E=2\), \(\dim\Ann(E)=3\), and linear fullness gives
\(n(F)\not\equiv0\) for \(0\ne n\in\Ann(E)\).  As \(n(F)\) has degree \(6\), we have $4\deg C\le6$ and therefore $\deg C=1.$

The line cannot meet the sphere.  Indeed, for every \(n\in\Ann(E)\), there is a quadratic $q_n$ such that $n(F)=c^4q_n.$ Hence at every \(p\in C\), we have $n(F)(p)=d(n(F))_p=0.$ Thus
$$
F(p),\ dF_p(T_p\PP^2)\subset E,
\qquad
\rank d[F]_p\le1.
$$
This contradicts Lemma~\ref{lem:boundary-rank} if
\(p\in C\cap\partial\B^2\), since there $\rank d[F]_p=2.$ Hence \(C\cap\partial\B^2=\varnothing\), so \(C\) is Hermitian negative.

Using the coordinates and factorization conventions of
\eqref{eq:negcoords}--\eqref{eq:F-on-C}, with \(E\) in place of \(U\), write
$$
C=(x=0),\qquad F|_C=\chi Y,\qquad
Y=(Y_1,Y_2)\in E,\qquad \gcd(Y_1,Y_2)=1.
$$

The Hermitian form on \(E\) is nondegenerate by the two-plane polarity argument.  As in the proof of Lemma~\ref{lem:projection}, orthogonal projection onto \(E\) gives $\pi_EF=\omega Y.$ Together with the \(x^4\)-divisibility of the \(E^\perp\)-components, this yields
\begin{equation}\label{eq:highrankoneform}
F=\omega Y+x^4H,\qquad H\in E^\perp,
\end{equation}
where $\omega=\omega(Z)$ is a scalar homogeneous polynomial of degree
$6-\deg Y$ and $H=H(Z)$ is an $E^\perp$-valued homogeneous quadratic
polynomial in $Z$.  Write $(Y,Y^*)_{1,4}=\sigma A$, where
$A=A(z,w)$ is a scalar bihomogeneous polynomial of bidegree
$(\deg Y-1,\deg Y-1)$.  Substituting \eqref{eq:highrankoneform} into the incidence identity gives
\[
\omega\omega^*A+\sigma^3(H,H^*)_{1,4}=0.
\]
Restricting to $x=y=0$ shows $\sigma\mid A$, so $\sigma^2\mid(Y,Y^*)_{1,4}$.  The argument following \eqref{eq:corank-one-double-div} shows that,
at a general polar pair, \(Y^*\) and its first derivative are dependent. Hence $d[Y]=0$ generically, contradicting nonconstancy.
\end{proof}

\begin{lemma}\label{lem:fourth-contact}
There is no component $C$ on which $\Phi|_\Gamma$ has generic rank two and $d[F]$ has generic rank two.
\end{lemma}
\begin{proof}
Following the argument of Lemma~\ref{lem:higher-rank-one-rigidity}, we first show that
\(\deg C=1\).  For general \(z\in C\), set
\[
E_z:=\Span\{F(z),dF_z(T_z\PP^2)\}.
\]
Since \(d[F]\) has generic rank two along \(C\), \(\dim E_z=3\).
The rank-two assumption on \(\Phi|_\Gamma\) implies, for general \(v\),
\[
\partial_v^2F(z),\ \partial_v^3F(z)
\in\Span\{F(z),\partial_vF(z)\}\subset E_z.
\]
Thus all mixed second and
third derivatives of \(F\) lie in \(E_z\).  Hence differentiation along \(C\)
preserves \(E_z\), so \(E_z\) is constant on a dense open subset of \(C\);
write it as \(E\subset\CC^5\).

For every \(n\in\Ann(E)\), \(n(F)\) and all of its derivatives of order at
most three vanish along \(C\).  Thus
\[
c^4\mid n(F).
\]
Since \(\dim E=3\), \(\dim\Ann(E)=2\), and linear fullness gives
\(n(F)\not\equiv0\) for \(0\ne n\in\Ann(E)\).  If \(\deg C\ge2\), then
\(\deg c^4>6\), which is impossible for the sextic \(n(F)\).  Therefore
\(\deg C=1\).

We next identify the third-order projective contact along this line.  Choose
a linear projection \(\pi:\CC^5\to E\) with \(\pi|_E=\mathrm{id}_E\).
Near a general point of \(C\), choose affine source coordinates
\((\xi,\eta)\) with \(C=(\eta=0)\), and, after multiplying the projected lift
by a local unit, write a local affine representative of \([\pi F]\) as
\[
(1,h),\qquad h=(h_1,h_2).
\]
Since $\Phi|_\Gamma$ has generic rank two along $C$,  we have along \(C\)
\begin{equation}\label{eq:fourth-contact-directional}
\partial_v^2h\wedge\partial_vh=0,
\qquad
\partial_v^3h\wedge\partial_vh=0
\end{equation}
for every \(v=a\partial_\xi+b\partial_\eta\).  Moreover
\(h_\xi,h_\eta\) are independent because \(d[F]\) has rank two.

Comparing the coefficients of \(a^3,a^2b,ab^2,b^3\) in
$$
\bigl(a^2h_{\xi\xi}+2ab h_{\xi\eta}+b^2h_{\eta\eta}\bigr)
\wedge
\bigl(a h_\xi+b h_\eta\bigr)=0
$$
gives scalar holomorphic functions \(\alpha=\alpha(\xi)\) and
\(\mu=\mu(\xi)\) such that
\begin{equation}\label{eq:fourth-contact-second}
h_{\xi\xi}=2\alpha h_\xi,
\qquad
h_{\xi\eta}=\mu h_\xi+\alpha h_\eta,
\qquad
h_{\eta\eta}=2\mu h_\eta.
\end{equation}
Differentiating these identities in the \(\xi\)-direction gives
\begin{align*}
h_{\xi\xi\xi}
&=(2\alpha_\xi+4\alpha^2)h_\xi,\
h_{\xi\xi\eta}
&=(\mu_\xi+3\alpha\mu)h_\xi
+(\alpha_\xi+\alpha^2)h_\eta,\
h_{\xi\eta\eta}
&=2\mu^2h_\xi+(2\mu_\xi+2\alpha\mu)h_\eta.
\end{align*}
Since $h_\eta$ and $h_\xi$ are linearly independent, we may write
$$
h_{\eta\eta\eta}=r h_\xi+s h_\eta.
$$
Comparing the coefficients in the cubic identity of
\eqref{eq:fourth-contact-directional} now yields
$$
\alpha_\xi=\alpha^2,
\qquad
\mu_\xi=\alpha\mu,
\qquad
r=0,
\qquad
s=6\mu^2.
$$

In particular,
\begin{equation}\label{eq:fourth-contact-third}
h_{\eta\eta\eta}=6\mu^2h_\eta.
\end{equation}

The first identity in \eqref{eq:fourth-contact-second} shows that the tangent
direction of \(h(C)\) is constant, so \(h(C)\) is contained in a projective
line.  If \(u\) is an affine coordinate on this line over the chart under
consideration, then \(u''/u'=2\alpha\), and therefore
$$
S(u)=\frac{u'''}{u'}-\frac32\left(\frac{u''}{u'}\right)^2
=2(\alpha_\xi-\alpha^2)=0.
$$

Since \(u\) is a rational function on \(C\simeq\PP^1\), it is fractional
linear.  After a projective change of coordinates on \(\PP(E)\), we may thus
take $h(\xi,0)=(\xi,0).$nThen \(\alpha=0\), \(\mu\) is constant, and
$$
h_\eta(\xi,0)=(\mu\xi+A,B),
\qquad B\ne0,
$$
for constants \(A,B\in\CC\).  Equations
\eqref{eq:fourth-contact-second}--\eqref{eq:fourth-contact-third} also give
$$
h_{\xi\xi\eta}=0,
\qquad
h_{\xi\eta\eta}=2\mu^2h_\xi,
\qquad
h_{\eta\eta\eta}=6\mu^2h_\eta.
$$
Consequently \(h\) has the same third jet along \(C\) as
$$
(\xi,\eta)\longmapsto
\left(
\frac{\xi+A\eta}{1-\mu\eta},
\frac{B\eta}{1-\mu\eta}
\right).
$$
After clearing the denominator, a homogeneous lift of this map is
$$
(1-\mu\eta,\,\xi+A\eta,\,B\eta),
$$
so it is induced by a linear injection \(L:\CC^3\to\CC^5\) with image
\(E\).  Since the components normal to \(E\) are divisible by \(c^4\), the
equality of the projected third jets shows that, on a dense open subset of
\(C\), every component of \(F\wedge LZ\) vanishes to order at least four.
As \(c\) is irreducible, this is equivalent to
\begin{equation}\label{eq:fourth-contact-wedge}
F\wedge LZ\equiv0\pmod{c^4}.
\end{equation}

Let \(4C\subset\PP^2\) be the fourth infinitesimal neighborhood defined by
\(c^4=0\).  Since \(L\) is injective, the coordinates of \(LZ\) have no
common zero.  On an open set where one coordinate of \(LZ\) is nonzero,
divide the corresponding coordinate of \(F\) by it.  The wedge relation
\eqref{eq:fourth-contact-wedge} shows that these local quotients agree on
overlaps, and hence define a section
$$
\overline P\in H^0(4C,\OO_{4C}(5)).
$$

Because \(C\) is a line, the exact sequence
$$
0\to\OO_{\PP^2}(1)\xrightarrow{\ c^4\ }
\OO_{\PP^2}(5)\to\OO_{4C}(5)\to0
$$
and \(H^1(\PP^2,\OO(1))=0\) lift \(\overline P\) to a homogeneous quintic
\(P=P(Z)\).  Therefore
\begin{equation}\label{eq:fourth-contact-normal-form}
F(Z)=P(Z)LZ+c(Z)^4\mathcal A(Z),
\end{equation}
where \(\mathcal A\) is a \(\CC^5\)-valued homogeneous quadratic polynomial.
Reducedness implies \(c\nmid P\), for otherwise every coordinate of \(F\)
would be divisible by \(c\).

By \eqref{eq:fourth-contact-normal-form}, every coordinate of \(F\) lies in
\((P,c^4)\).  Since \(c\nmid P\), this is a zero-dimensional complete
intersection of length \(5\cdot4=20\), and it is contained in the base scheme
\(Z_F\).  Lemma~\ref{lem:beta} therefore gives $\beta=\length(Z_F)\ge20,$ contradicting \eqref{eq:budget}.
\end{proof}

\subsection{Completion of the local estimate}\label{subsec:completion}
If $\Phi|_\Gamma$ has corank one, \eqref{eq:corank-one-complete} gives the desired estimate.  In higher corank, projective rank zero is impossible, projective rank one is excluded by Lemma~\ref{lem:higher-rank-one-rigidity}, and projective rank two is excluded by Lemma~\ref{lem:fourth-contact}.  Therefore every determinant component satisfies $\delta_C\ge3\ord_C\lambda.$ This proves Proposition~\ref{prop:jet-defect}.

\appendix

\section{Boundary-line exclusion in the corank-one case}\label{app:boundary-line-exclusion}

Retain the reduced degree-six lift $F$, the polarized identity, and the exceptional corank-one rank-two branch of Section~\ref{sec:local-proof}. We use only the pure-trace identity \eqref{eq:pure-trace} and the conclusions of Lemma~\ref{lem:pure-trace-reduction}: $C=(c=0)$ is a projective line and there are fixed spaces
\[
U\subset V\subset\CC^5,
\qquad
\dim U=2,
\qquad
\dim V=3,
\]
with
\begin{equation}\label{eq:A1}
F|_C\in U,
\quad
\Span\{F(z),dF_z(T_z\PP^2)\}=V\quad(z\in C\text{ general}),
\quad
c^3\mid\vartheta(F)\quad(\vartheta\in\Ann(V)).
\end{equation}
By Lemma~\ref{lem:no-contracted-curve}, $[F]|_C$ is nonconstant.

\begin{proposition}\label{prop:A-boundary-line-exclusion}
Under these hypotheses,
\[
C\cap\partial\B^2=\varnothing.
\]
Consequently $C$ is Hermitian negative.
\end{proposition}

\begin{proof}
Suppose that \(C\) meets the source sphere at \(p\). By Lemma~\ref{lem:boundary-rank}, choose a nonzero $(1,0)$ tangent vector \(X\) and a transverse direction \(n\) satisfying the three identities there. Together with \eqref{eq:A1}, this shows that \(V\) is nondegenerate of signature \((1,2)\), and hence \(V^\perp\) is negative definite. The pure-trace relation \eqref{eq:pure-trace} is invariant under changes of local lift and projective source coordinates, with the corresponding change of the one-form $\varrho$.

Every projective line meeting the source sphere is either tangent or secant. We exclude the two cases separately.

\textbf{The tangent case.}
Suppose $C$ is tangent to the source sphere at $p$. Choose projective Siegel coordinates centered at $p$ and $f(p)$ so that
\[
\partial\mathbb H_2=\{(z,w):\operatorname{Im}w=|z|^2\},
\qquad
C=(w=0),
\]
and choose target Siegel coordinates $(\psi,\phi)$, with
\[
\phi=(\phi_1,\phi_2,\phi_3)\in\CC^3,
\qquad
\partial\mathbb H_4=\{(\psi,\phi):\operatorname{Im}\psi=|\phi|^2\},
\]
adapted to the flag $U\subset V$ so that
\[
\PP(V)=(\phi_2=\phi_3=0),
\qquad
\PP(U)=(\psi=\phi_2=\phi_3=0).
\]
In these coordinates write the local map as
\[
f=(g,f_1,f_2,f_3),
\qquad
g=\psi\circ f,
\qquad
f_j=\phi_j\circ f.
\]
By \eqref{eq:A1}, there is a \(\mathbb C^2\)-valued holomorphic function $K$ near the origin such that
\begin{equation*}
(f_2,f_3)=w^3K,
\qquad
g(z,0)=0.
\end{equation*}
Define $s(z):=f_1(z,0).$ From now on, $(\cdot)'$ denotes differentiation with respect to $z$. Using boundary transversality and target dilations/rotations, we normalize 
\[
s(0)=0,
\qquad
s'(0)=g_w(0)=1.
\]
The complexified sphere identity is
\begin{equation}\label{eq:A2}
g(z,w)-g^*(\zeta,\omega)
=2i\sum_{j=1}^3 f_j(z,w)f_j^*(\zeta,\omega),
\qquad
w-\omega=2iz\zeta.
\end{equation}
Setting $\omega=0$ and differentiating in $\zeta$ at zero gives
\begin{equation}\label{eq:A3}
g(z,2iz\zeta)=2if_1(z,2iz\zeta)s^*(\zeta),
\qquad
g_w(z,0)=\frac{s(z)}{z}.
\end{equation}
Let \(\tau:=\frac{\partial}{\partial z}\) along \(C\). Applying the pure-trace identity \eqref{eq:pure-trace}, with \(\varrho\) as defined there, to the tangent--tangent and tangent--normal pairs gives
\[
s''=2\varrho(\tau)s',
\qquad
(g_w)'=\varrho(\tau)g_w.
\]
Using \eqref{eq:A3}, we have
\[
\frac{s''}{2s'}=\frac{s'}{s}-\frac1z.
\]
Hence there is a constant $b$ such that
\begin{equation}\label{eq:A4}
s(z)=\frac{z}{1+bz}.
\end{equation}

Put $\eta_0:=-ib/2$ and, in the target coordinates $(\psi,\phi_1,\phi_2,\phi_3)$, the transformation
\[
(\psi,\phi_1,\phi_2,\phi_3)\longmapsto
\frac{1}{1-2i\eta_0\phi_1-i|\eta_0|^2\psi}\left(
\psi,
\phi_1+\eta_0\psi,
\phi_2,
\phi_3
\right)
\]
preserves the Siegel domain, fixes the origin, sends \eqref{eq:A4} to $s(z)=z$, and preserves $(f_2,f_3)\in(w^3)$. Hence assume $s(z)=z$.

Equation \eqref{eq:A2} with $\omega=0$ now implies $zg=wf_1$, so there is a holomorphic function $A=A(z,w)$ near the origin such that
\[
f_1=zA,
\qquad
g=wA.
\]
Since $f_1(z,0)=s(z)=z$, we have $A(z,0)=1.$
In particular, $A$ is a holomorphic unit near the origin. Since
$A^{-1}(z,0)=1$, there is a holomorphic function $h=h(z,w)$ such that
\[
A^{-1}=1+wh.
\]
Set $L:=K/A,$ which is again a $\CC^2$-valued holomorphic function near the origin.
Dividing \eqref{eq:A2} by $AA^*$ and cancelling $w\omega$ in the local
integral incidence ring gives
\begin{equation}\label{eq:A5}
h^*(\zeta,\omega)-h(z,w)
=
2iw^2\omega^2L(z,w)\cdot L^*(\zeta,\omega).
\end{equation}
Setting $\omega=0$ shows that $h$ is constant; evaluation at the origin
shows that this constant is real. Hence \eqref{eq:A5} gives $L\cdot L^*=0$ on the complexified sphere. So
$L\equiv0$, and hence $K=0$. Thus the local image is contained in
$\PP(V)$. By rational continuation, the global image is contained in
$\PP(V)$, contradicting linear fullness. Therefore the tangent case is
impossible.

\vspace{.5cm}
\textbf{The secant case.}
Now suppose \(C\) is secant. In the standard affine chart of \(\mathbb P^2\), the affine line represented by \(C\) meets the ball. Hence the source Hermitian form on the two-plane representing \(C\) has signature \((1,1)\), and therefore the form $\langle ,\rangle_{1,4}$ on \(U\) also has signature \((1,1)\).

\smallskip
\noindent\emph{Common reduction.}
Following the adapted-coordinate convention of Section~\ref{sec:local-proof} and the factorization conventions of
\eqref{eq:negcoords}--\eqref{eq:F-on-C}, write
\[
Z=(x,z_1,z_2),
\qquad
W=(y,w_1,w_2),
\qquad
z=(z_1,z_2),
\qquad
w=(w_1,w_2).
\]
In this case the binary Hermitian form has signature $(1,1)$, so
\[
(Z,W)_{1,2}=\sigma_-(z,w)-xy,
\qquad
\sigma_-(z,w)=z_1w_1-z_2w_2.
\]
Thus $C=(x=0)$,
\[
F|_C=\chi Y,
\qquad
Y=(Y_1,Y_2)\in U,
\qquad
\gcd(Y_1,Y_2)=1,
\qquad
\nu=\deg Y.
\]

Choose a nonzero $e_\perp\in V\cap U^\perp$. Since $V$ and $U$ have
signatures $(1,2)$ and $(1,1)$, respectively,
\[
\varepsilon:=\langle e_\perp,e_\perp\rangle_{1,4}<0.
\]
Write $[\partial_xF]_{V/U}=g_Ce_\perp$. Applying \eqref{eq:pure-trace} as in the computation leading to
\eqref{eq:Wronskian-square} gives
\begin{equation*}
g_C^2=\gamma\chi^2\Wr(Y),
\qquad
g_C=\chi\rho,
\qquad
\Wr(Y)=\gamma^{-1}\rho^2,
\qquad
\deg\rho=\nu-1.
\end{equation*}

In the affine chart $z_1=1$, the source disk $C\cap\B^2$ is the unit disk. Since $U$ has signature $(1,1)$, after choosing an adapted Hermitian basis of $U$, the target disk $\PP(U)\cap\B^4$ is also the standard unit disk. Since $F|_C=\chi Y$ and $f$ is proper, the restriction $Y_2/Y_1$ is a nonconstant rational proper self-map of the unit disk. Hence it is a finite Blaschke product of degree $\nu$.

By $\Wr(Y)=\gamma^{-1}\rho^2$, every ramification multiplicity of
$Y_2/Y_1$ is even. A finite Blaschke product has no ramification on the unit
circle, and its ramification points are paired by reflection across the
unit circle. Hence Riemann--Hurwitz gives $2\nu-2\equiv0\pmod4$ and $\nu\in\{1,3,5\}.$

Orthogonal projection onto the nondegenerate two-plane $U$ gives a homogeneous scalar polynomial $\widehat\chi$, a quintic $\widetilde g$, and a cubic $V^\perp$-valued function $H$ such that
\begin{equation}\label{eq:A6}
F=\widehat\chi Y+x\widetilde g\,e_\perp+x^3H,
\qquad
H\in V^\perp,
\qquad
\widetilde g(0,z)=\chi\rho.
\end{equation}

Choose an adapted Hermitian basis $e_+,e_-$ of $U$ such that
\[
\langle e_+,e_+\rangle_{1,4}=1,\qquad
\langle e_-,e_-\rangle_{1,4}=-1,\qquad
Y=Y_1e_++Y_2e_-.
\]
By Lemma~\ref{lem:no-base}, $F$ never vanishes on
$\overline{\B^2}$ and we also have $\langle F,F\rangle_{1,4}\geq0$ on $\overline{\B^2}$. Hence the factor $F_0=\widehat\chi Y_1$ is non-zero by \eqref{eq:A6}, and therefore
\begin{equation}\label{eq:A7}
\widehat\chi\ne0
\quad\text{on }\;
\overline{\B^2}.
\end{equation}

\smallskip
\noindent\emph{The cases $\nu=1$ and $\nu=3$.}
For $\nu=1$, disk automorphisms and rescaling give
\[
Y=(z_1,z_2),
\qquad
\rho=1.
\]
For $\nu=3$, the ramification divisor consists of a reflected pair of double points. Disk automorphisms send them to $0$ and $\infty$, and after rescaling
\[
Y=(z_1^3,z_2^3),
\qquad
\rho=z_1z_2.
\]
All these transformations are projective and preserve reduced degree.

Divide \eqref{eq:A6} by the nonvanishing function $\widehat\chi$ near the closed ball and expand only the $e_\perp$-coefficient:
\begin{equation*}
F/\widehat\chi
=Y+x\rho e_\perp+x^2\eta e_\perp+x^3(\beta e_\perp+N),
\qquad
N\in V^\perp.
\end{equation*}
Recall $\varepsilon:=\langle e_\perp,e_\perp\rangle_{1,4}<0$. At a general complexified polar pair, compare the terms of total degree two and three in $(x,y)$. For $\nu=1$ the leading coefficient gives $1+\varepsilon=0$; for $\nu=3$ it gives $3+\varepsilon=0$. In both cases the next order gives $\eta=0$. Thus
\begin{equation*}
(\nu,\varepsilon,\eta)=(1,-1,0)\quad\text{or}\quad(3,-3,0).
\end{equation*}
Let $\epsilon=0$ in degree one and $\epsilon=1$ in degree three. The full polarized identity becomes
\[
0=\epsilon x^3y^3
+\varepsilon(x^3y\,\beta\rho^*+xy^3\rho\beta^*)
+x^3y^3\bigl(\varepsilon\beta\beta^*+(N,N^*)_{1,4}\bigr).
\]
Cancel $xy$ and set $y=0$ at a general polar pair with $\rho^*(w)\ne0$. Then $\beta=0$, and hence
\begin{equation}\label{eq:A8}
(N(Z),N^*(W))_{1,4}=-\epsilon
\qquad\text{on }xy=\sigma_-.
\end{equation}
For $\nu=1$, the negative definiteness of $V^\perp$ and the real sphere force $N=0$, so the map is projectively linear. For $\nu=3$, by \eqref{eq:A7}, $N$ is holomorphic near the closed ball, and with the positive definite metric $-\langle\ ,\ \rangle_{1,4}$ on $V^\perp$, \eqref{eq:A8} says $\|N\|=1$ on the sphere. The maximum principle implies that $N$ either constant or a proper self-map of $\B^2$; in the latter case Alexander's theorem implies that it is a ball automorphism. Thus the original map has a homogeneous representation of degree at most four, contradicting reduced degree six. Therefore $\nu=1,3$ are impossible.

\smallskip
\noindent\emph{The case $\nu=5$.}
By \eqref{eq:A7}, the linear form $\widehat\chi$ has no zero on the closed
source ball, so its dual Hermitian norm is positive. A source transformation
preserving the disk form, a rotation of $x$, and a scalar rescaling put
\begin{equation*}
\widehat\chi=z_1+\alpha x,
\qquad
\chi=z_1,
\qquad
0\le\alpha<1.
\end{equation*}
Choose a negative orthonormal basis of $V^\perp$ and write
\[
H(Z)\cdot H^*(W)=-(H(Z),H^*(W))_{1,4}.
\]
Define a Hermitian kernel $A_4$ of bidegree $(4,4)$ by
\begin{equation}\label{eq:A9}
Y_1(z)Y_1^*(w)-Y_2(z)Y_2^*(w)
=\sigma_-(z,w)A_4(z,w).
\end{equation}
The polarized identity is
\begin{equation}\label{eq:A10}
\widehat\chi\widehat\chi^*A_4
+\varepsilon\widetilde g\widetilde g^*
=\sigma_-^2H\cdot H^*,
\qquad
xy=\sigma_-.
\end{equation}

Write $\widetilde g=\sum_{j=0}^5x^jg_j$ with
$g_0=\chi\rho$. Substitute $y=\sigma_-/x$ in \eqref{eq:A10}.
The highest Laurent powers give $g_5=g_4=0$, then
$g_3=g_2=0$ by irreducibility of $\sigma_-$. If $\alpha=0$, the next
coefficient gives $g_1=0$. If $\alpha\ne0$, take a general point
$Z\in(\widehat\chi=0)$ and $W=(w,0)$ with $\sigma_-(z,w)=0$;
equation \eqref{eq:A10} gives $\widetilde g(Z)=0$, so
$\widehat\chi\mid\widetilde g$. In either case,
\begin{equation}\label{eq:A11}
\widetilde g=\widehat\chi\rho,
\qquad
F=\widehat\chi(Y+x\rho e_\perp)+x^3H.
\end{equation}
Substitution into \eqref{eq:A10} gives a Hermitian kernel $Q$ of
bidegree $(2,2)$ such that
\begin{equation}\label{eq:A12}
A_4+\varepsilon\rho\rho^*=\sigma_-^2Q,
\qquad
H\cdot H^*=\widehat\chi\widehat\chi^*Q
\quad\text{on }xy=\sigma_-.
\end{equation}
Multiplying the first identity by $\sigma_-$ gives
\begin{equation}\label{eq:A13}
\sigma_-^3Q
=Y_1Y_1^*-Y_2Y_2^*
+\varepsilon(z_1\rho)(w_1\rho^*)
-\varepsilon(z_2\rho)(w_2\rho^*).
\end{equation}
Hence the separated rank of $\sigma_-^3Q$ is at most four.

The coefficient matrix of $A_4$ is positive definite. Indeed, after a
scalar normalization the degree-five Blaschke product can be written
\[
Y_1=\prod_{j=1}^5D_j,
\qquad
Y_2=e^{i\theta}\prod_{j=1}^5N_j,
\qquad
D_j=z_1-a_jz_2,
\qquad
N_j=z_2-\bar a_jz_1,
\qquad
|a_j|<1.
\]
Substituting into \eqref{eq:A9} gives
\begin{equation*}
A_4=\sum_{j=1}^5\Pi_j\Pi_j^*,
\qquad
\Pi_j=\sqrt{1-|a_j|^2}
\prod_{k<j}N_k\prod_{k>j}D_k,
\end{equation*}
and the five quartics $\Pi_j$ are independent. The matrix of
$A_4+\varepsilon\rho\rho^*$ therefore has at least four positive
eigenvalues. It is not positive semidefinite, since by \eqref{eq:A12}
its diagonal vanishes on $|z_1|=|z_2|$ while a positive-semidefinite
Hermitian kernel is a sum of squares. Hence the separated rank of
$\sigma_-^2Q$ is five. Write
\begin{equation}\label{eq:A14}
H=H_0+xH_1+x^2H_2+x^3u,
\qquad
\deg H_j=3-j.
\end{equation}
The polynomial $H_0$ is nonzero. Indeed, if $H_0=0$, then restricting
the second identity in \eqref{eq:A12} to $x=y=0$ gives
$\sigma_-\mid Q$. Hence $\sigma_-^4$ divides $\sigma_-^3Q$, and
Lemma~\ref{lem:seprank}(ii) would imply that $\sigma_-^3Q$ has separated rank at
least five, contradicting the preceding upper bound of four obtained from \eqref{eq:A13}. 

We now substitute \(y=\sigma_-/x\) into the equality $H\cdot H^*=\widehat\chi\widehat\chi^*Q$ in \eqref{eq:A12} and compare the coefficients of each Laurent power of \(x\) on both sides in the following cases. We will prove each case implies that $\widehat\chi\mid H$.

\emph{Case 1: $u\ne0$.}
The $x^3$ coefficient gives $u\cdot H_0^*=0$. Choose a unit vector $n$
orthogonal to $u$ and write $H_0=nP_3$, $P_3\ne0$. Write
$H_1=nP_2+uB_2$ and $H_2=nA_1+uP_1$. With $\mu=\|u\|^2$, the
$x^2$ coefficient is
\[
A_1P_3^*+\mu\sigma_-B_2^*=0.
\]
Since $\sigma_-$ is irreducible and cannot divide a nonzero separated
product, $A_1=B_2=0$. Hence
\[
H=n(P_3+xP_2)+ux^2(P_1+x).
\]
The $x$ and constant coefficients are
\begin{align}
P_2P_3^*+\mu\sigma_-^2P_1^*
&=\alpha\chi^*Q,\label{eq:A15}\\
P_3P_3^*+\sigma_-P_2P_2^*
+\mu\sigma_-^2P_1P_1^*+\mu\sigma_-^3
&=(\chi\chi^*+\alpha^2\sigma_-)Q.\label{eq:A16}
\end{align}
Multiply \eqref{eq:A15} by $\chi$, subtract $\alpha$ times
\eqref{eq:A16}, and reduce modulo $\sigma_-$. Since $\sigma_-$ cannot
divide a nonzero separated product, this gives
\[
\chi P_2=\alpha P_3.
\]

If $\alpha>0$, substitution into \eqref{eq:A15} gives
$\chi\mid P_1$; writing $P_1=\kappa\chi$ and substituting in
\eqref{eq:A16} gives $\kappa=1/\alpha$. Thus
\[
H=\frac{\widehat\chi}{\alpha}(nP_2+ux^2).
\]

If $\alpha=0$, then $P_2=P_1=0$. Evaluation of \eqref{eq:A16} at
$w=(0,1)$ gives $P_3=\lambda z_2^3$ with $|\lambda|^2=\mu$. Setting
$a=z_1w_1$ and $b=z_2w_2$, we obtain
\[
\sigma_-^3Q
=\mu(a-b)^3(a^2-3ab+3b^2)
=\mu(a^5-6a^4b+15a^3b^2-19a^2b^3+12ab^4-3b^5).
\]
Thus $\sigma_-^3Q$ has six nonzero diagonal coefficients in the
degree-five monomial basis, contradicting the preceding upper bound of
four on its separated rank proved after \eqref{eq:A13}.

\emph{Case 2: $u=0$ and $H_2\ne0$.}
The $x^2$ coefficient gives $H_2\cdot H_0^*=0$. Since both terms are
nonzero, their coefficient spans are orthogonal lines. Choose
orthonormal vectors $n,v$ and write
\[
H_0=nP_3,
\qquad
H_2=vP_1,
\qquad
H_1=nA_2+vB_2.
\]
The $x$ and constant coefficients are
\begin{align}
A_2P_3^*+\sigma_-P_1B_2^*
&=\alpha\chi^*Q,\label{eq:A17}\\
P_3P_3^*+\sigma_-(A_2A_2^*+B_2B_2^*)
+\sigma_-^2P_1P_1^*
&=(\chi\chi^*+\alpha^2\sigma_-)Q.\label{eq:A18}
\end{align}
Reduction modulo $\sigma_-$ gives
\[
\chi A_2=\alpha P_3.
\]

If $\alpha=0$, then $A_2=B_2=0$. Put
$p=P_3(0,1)$ and $q=P_1(0,1)$. Equation \eqref{eq:A18} at
$w=(0,1)$ gives
\[
\bar pP_3+\bar qz_2^2P_1=0.
\]
Its $z_2^3$ coefficient is $|p|^2+|q|^2$, so $p=q=0$. Hence
$\chi\mid P_3,P_1$, and $\widehat\chi=\chi$ divides $H$. 

If
$\alpha>0$, write $P_3=\chi R_2$, $A_2=\alpha R_2$; equation
\eqref{eq:A17} gives $B_2=\chi R_1$ and
\[
Q=R_2R_2^*+\frac{\sigma_-}{\alpha}P_1R_1^*.
\]
Hermitian symmetry implies $R_1=\kappa P_1$ for a real constant
$\kappa$, and \eqref{eq:A18} gives $\kappa=1/\alpha$. Hence
\[
H=\widehat\chi\left(nR_2+\frac{x}{\alpha}vP_1\right).
\]

\emph{Case 3: $u=H_2=0$.}
Then $H=H_0+xH_1$ and
\begin{align}
H_1\cdot H_0^*
&=\alpha\chi^*Q,\label{eq:A19}\\
H_0\cdot H_0^*+\sigma_-H_1\cdot H_1^*
&=(\chi\chi^*+\alpha^2\sigma_-)Q.\label{eq:A20}
\end{align}

If $\alpha=0$ and $H_1=0$, evaluation at $z=w=(0,1)$ gives
$H_0(0,1)=0$, hence $\chi\mid H_0$. Suppose $H_1\ne0$. The two
coefficient spans are orthogonal lines, so write
$H_0=nP_3$, $H_1=vP_2$ with $n\perp v$. Then
\[
P_3P_3^*+\sigma_-P_2P_2^*=\chi\chi^*Q.
\]
Put $p=P_3(0,1)$ and $q=P_2(0,1)$. At $w=(0,1)$ this becomes
$\bar pP_3-\bar qz_2P_2=0$. If $p=q=0$, then $\chi$ divides both
scalar factors. Otherwise $|p|=|q|\ne0$, so
$P_3=\lambda z_2P_2$ with $|\lambda|=1$, and the last identity gives
$Q=P_2P_2^*$. Then $Q$ has separated rank one, so
$\sigma_-^2Q$ has separated rank at most three, contradicting the separated-rank-five conclusion established before \eqref{eq:A14}. Thus $\widehat\chi=\chi$ divides
$H$.

Assume $0<\alpha<1$. If $H_0(0,1)\ne0$, \eqref{eq:A19} at
$w=(0,1)$ makes the coefficient span of $H_1$ at most a line. If
$H_1=0$, then $Q=0$; otherwise, writing $H_1=vP_2$, equation
\eqref{eq:A19} shows $Q=P_2R^*$ for a quadratic $R$. In either case
$Q$ has separated rank at most one, so $\sigma_-^2Q$ has separated
rank at most three, again a contradiction. Hence
\[
H_0=\chi K_2
\]
for a vector quadratic $K_2$. Dividing \eqref{eq:A19} by $\chi^*$
gives $H_1\cdot K_2^*=\alpha Q.$
The identity \eqref{eq:A20} gives
\[
\chi\chi^*(Q-K_2\cdot K_2^*)
=\sigma_-(H_1\cdot H_1^*-\alpha^2Q).
\]
Since $\gcd(\sigma_-,\chi\chi^*)=1$, we may define
\[
T=\frac{Q-K_2\cdot K_2^*}{\sigma_-},
\qquad
D=H_1-\alpha K_2.
\]
Then
\begin{align}
D\cdot D^*&=(\chi\chi^*-\alpha^2\sigma_-)T,\label{eq:A21}
\\
D\cdot K_2^*&=\alpha\sigma_-T.\label{eq:A22}
\end{align}
The Hermitian kernel
\[
\chi\chi^*-\alpha^2\sigma_-
=(1-\alpha^2)z_1w_1+\alpha^2z_2w_2
\]
is positive definite. Hence, restricting \eqref{eq:A21} to the diagonal shows that $T$ is positive semidefinite. 

If $T$ has
separated rank two, the term $(\chi\chi^*-\alpha^2\sigma_-)T$ has rank three on binary
quadratics, whereas the term $D\cdot D^*$ has rank at most two, a
contradiction. 

If $T$ has separated rank one, write $T=t(z)t^*(w)$. Since
the term $(\chi\chi^*-\alpha^2\sigma_-)T$ is positive definite of rank two, \eqref{eq:A21}
gives $D=tM$ with an invertible vector-linear map $M$. Cancelling
$t$ in \eqref{eq:A22} and using the invertibility of $M$ gives
$t\mid K_2$. Thus $t(z)$ is a factor of $Q$, and all left quartic
factors of $\sigma_-^2Q$ lie in the four-dimensional space of
quartics divisible by $t$. This contradicts the previously established fact that $\sigma_-^2Q$ has separated rank five.

Hence $T=0$, and the first
identity gives $D=0$. Therefore
\[
H=\chi K_2+\alpha xK_2=\widehat\chi K_2.
\]
The three cases exhaust \eqref{eq:A14}, so
$\widehat\chi\mid H$.

Equation \eqref{eq:A11} then makes every coordinate of $F$ divisible by
the nonconstant linear form $\widehat\chi$, contradicting reducedness.
Thus the secant case is impossible.

Both possibilities have been excluded, so $C\cap\partial\B^2=\varnothing$. The source Hermitian form on the two-plane representing $C$ is therefore definite. Since the ambient positive index is one, the restriction cannot be positive definite; hence it is negative definite.
\end{proof}

\end{document}